\documentclass[11pt]{article}
\usepackage[top=.9in,bottom=.9in,right=.9in,left=.9in]{geometry}
\usepackage{amssymb,amsmath,amsthm}
\usepackage{arydshln}
\usepackage{verbatim}
\usepackage{titlesec}
\usepackage{subfigure}
\usepackage[linesnumbered,ruled]{algorithm2e}
\usepackage{bm}
\usepackage{pdfpages}
\usepackage[round]{natbib}
\definecolor{forestgreen}{rgb}{0.13, 0.55, 0.13}
\usepackage[colorlinks, linkcolor=blue,citecolor=forestgreen,urlcolor = black, bookmarks=true]{hyperref}

\usepackage[T1]{fontenc}
 
\usepackage[nameinlink, noabbrev, capitalize]{cleveref}
\usepackage{times}

\allowdisplaybreaks

\newcommand{\w}{\bm{w}}
\newcommand{\Var}{\mathrm{Var}}
\newcommand\blfootnote[1]{%
  \begingroup
  \renewcommand\thefootnote{}\footnote{#1}%
  \addtocounter{footnote}{-1}%
  \endgroup
}

\crefname{equation}{}{}
\crefrangeformat{equation}{(#3#1#4) --~(#5#2#6)}
\crefmultiformat{equation}{(#2#1#3)}{, (#2#1#3)}{, (#2#1#3)}{, (#2#1#3)}
\crefname{lem}{Lemma}{Lemmas}
\crefname{section}{Section}{Sections}
\crefname{subsubsubsection}{Section}{Sections}
\crefname{rem}{Remark}{Remarks}
\crefname{cor}{Corollary}{Corollaries}
\crefname{figure}{Figure}{Figures}
\crefname{table}{Table}{Tables}
\Crefname{lem}{Lemma}{Lemmas}
\Crefname{line}{Line}{Lines}
\Crefname{fact}{Fact}{Facts}
\crefname{thm}{Theorem}{Theorems}
\crefname{assumption}{Assumption}{Assumptions}

\usepackage{commath}

\newtheorem{thm}{Theorem}[section]
\newtheorem{defn}[thm]{Definition}
\newtheorem{lem}[thm]{Lemma}

\newtheorem{cor}[thm]{Corollary}
\newtheorem{prop}[thm]{Proposition}
\newtheorem{assumption}{Assumption}
\newtheorem{fact}[thm]{Fact}
\newtheorem{claim}[thm]{Claim}

\newtheorem{rmk}{Remark}[section]

\title{Sample-Efficient Linear Regression with Self-Selection Bias\blfootnote{J.G. is supported by Vannevar Bush Faculty Fellowship ONR-N00014-20-1-2826 and Simons Investigator Award 622132. E.M. is supported in part by Vannevar Bush Faculty Fellowship ONR-N00014-20-1-2826, Simons Investigator Award 622132, Simons-NSF DMS-2031883,  and NSF Award CCF 1918421.}}
\author{Jason Gaitonde\\
Massachusetts Institute of Technology\\
\texttt{gaitonde@mit.edu}
\and Elchanan Mossel\\
Massachusetts Institute of Technology\\
\texttt{elmos@mit.edu}}

\begin{document}
\maketitle

\begin{abstract}%
  We consider the problem of \emph{linear regression with self-selection bias} in the unknown-index setting, as introduced in recent work by Cherapanamjeri, Daskalakis, Ilyas, and Zampetakis [STOC 2023]. In this model, one observes $m$ i.i.d. samples $(\bm{x}_{\ell},z_{\ell})_{\ell=1}^m$ where $z_{\ell}=\max_{i\in [k]}\{\bm{x}_{\ell}^T\w_i+\eta_{i,\ell}\}$, but the maximizing index $i_{\ell}$ is unobserved. Here, the covariates $\bm{x}_{\ell}$ are assumed to be distributed as $\mathcal{N}(0,I_n)$ and the noise distribution $\bm{\eta}_{\ell}\sim \mathcal{D}$ is centered and independent of $\bm{x}_{\ell}$. Under natural assumptions on the geometry of the $\w_1,\ldots,\w_k$, and when $\bm{\eta}$ is known to be $\mathcal{N}(0,I_k)$, \citet{DBLP:conf/stoc/CherapanamjeriD23} provide an algorithm to recover $\w_1,\ldots,\w_k$ up to additive $\ell_2$-error $\varepsilon$ with sample complexity $\tilde{O}(n)\cdot \exp(\mathsf{poly}(1/\varepsilon)\cdot \widetilde{O}(k))$ and time complexity $\mathsf{poly}(n)\cdot \exp(\mathsf{poly}(1/\varepsilon)\cdot \widetilde{O}(k))$.

  We provide a novel and near optimally sample-efficient (in terms of $k$) algorithm for this problem with \emph{polynomial} sample complexity $\tilde{O}(n)\cdot \mathsf{poly}(k,1/\varepsilon)$ and significantly improved time complexity $\mathsf{poly}(n,k,1/\varepsilon)+O(\log(k)/\varepsilon)^{O(k)}$. When $k=O(1)$, our algorithm has $\mathsf{poly}(n,1/\varepsilon)$ runtime, generalizing the fully-polynomial time guarantee of an explicit moment matching algorithm of \citet{DBLP:conf/stoc/CherapanamjeriD23} for the special case of $k=2$. Our algorithm succeeds under significantly relaxed noise assumptions, and therefore also recovers the regressors in the related setting of \emph{max-linear regression} where the added noise is taken outside the maximum. For this problem, our algorithm is efficient in a much larger range of $k$ than the state-of-the-art due to Ghosh, Pananjady, Guntuboyina, and Ramchandran \cite[IEEE Trans. Inf. Theory 2022]{ghosh} for not too small $\varepsilon$, and leads to improved algorithms for any $\varepsilon$ by providing a warm start for existing local convergence methods. Our approach relies on an intricate analysis of the relationship between conditional low-order moments of the observations with the unknown geometry of the regressors to algorithmically distinguish close vectors from $\varepsilon$-far vectors to all $\w_i$. This improved analysis in terms of `simple statistics' is essential for obtaining near-optimal sample complexity in $k$.
\end{abstract}

\newpage
\section{Introduction}
A common feature of many real-world learning problems, particularly those that arise as a result of an economic, strategic, or otherwise private process, is that the observed outcomes are \emph{endogenously} aggregated from a small number of unobserved, simpler functions of the observed data. For instance, in deciding which occupation to pursue, individuals may assess their aptitude among possible professions based on their known capabilities and choose the occupation that maximizes their earnings. In medical settings, doctors may prescribe different treatments depending on the features of a patient by evaluating which treatment is expected to induce the best outcome based on their expert knowledge of the relationship between features and treatment effectiveness separately. In both of these cases, na\"ively using the observed data for each occupation or treatment to infer the fundamental relationship between covariates and the efficacy of each separate option may lead to biased estimates because the choice of option is itself correlated with the covariates (\citet*{roy,NBERw0249,heckman_econometrica}).

This phenomenon has been extensively observed in the economics and econometrics literature, and is commonly referred to as \emph{self-selection bias}; the observed outcome is selected in a systematic way among the underlying, simpler models, so that the separate possible outcomes nontrivially interact to produce the final observation. We refer to \citet*{Heckman2017} for a comprehensive discussion of different economic models and self-selection procedures with these properties.

Even in relatively simple settings, where the selection procedure is known to be the \emph{maximum} among simple functions of the data, a key feature of many such settings is that we often cannot even obtain information about which unknown function produced the final output. This unknown-index setting makes learning the parameters of the underlying models rather challenging, since one does not obtain any direct information about any individual latent model. But this feature is of particular relevance in applications where obtaining information about the identity of the maximizing index is expensive, or perhaps even impossible to obtain from existing data. For instance, an existing data set with demographic features of individuals along with their earnings may not list occupation (or some other postulated latent process) as a categorical feature; in this case, even though the known data is clearly valuable towards inferring the general relationship between features and earnings from different occupations, one cannot use knowledge of chosen professions in attempting to discover this correspondence. Similarly, in learning the general relationship between item features and value from observed auction data among a small number of individuals (\citet*{athey,DBLP:conf/sigecom/CherapanamjeriD22}), or from market prices when supply and demand are out of equilibrium (\citet*{fair_jaffee,amemiya}), one may only observe the price while the identity of the winning bidder herself or whether supply exceeds demand remains hidden. In other applications, like the medical scenarios above, the identity of the maximizing index (treatment undertaken) may even have been intentionally obfuscated due to privacy concerns.

\subsection{Problem Formulation}
In this paper, we focus on a general version of the problem of learning \emph{linear} functions in the presence of self-selection bias in the unknown-index setting, where the selection function is taken to be the maximum function and each linear model is perturbed with separate noise components. This model was recently introduced by \citet*{DBLP:conf/stoc/CherapanamjeriD23}; while recovering a linear model in isolation is quite classical, the problem of recovering the underlying linear models under this simple form of self-selection already induces a number of algorithmic challenges. 

More formally, in the general model of linear regression with self-selection bias model we consider, we observe i.i.d. observations $(\bm{x}_{\ell},z_{\ell})\in \mathbb{R}^n\times \mathbb{R}$ generated by the following process:
\begin{enumerate}
    \item First, the data $\bm{x}_{\ell}\sim \mathcal{N}(0,I_n)$ is sampled according to a standard, isotropic Gaussian.

    \item The final observation $z_{\ell}$ is generated as the maximum of $k$ different, unknown linear functions of $\bm{x}$ that are \emph{separately} perturbed by centered noise $\bm{\eta}_{\ell}\in \mathbb{R}^k\sim \mathcal{D}$ that is independent of $\bm{x}_{\ell}$: that is,
    \begin{equation}
    \label{eq:lin_reg_ss}
        z_{\ell} = \max_{j\in [k]}\left\{\bm{x}_{\ell}^T\w_j+\eta_{\ell,j}\right\}\tag{LR-SSB},
    \end{equation}
    where $\w_1,\ldots,\w_k\in \mathbb{R}^n$ are unknown, fixed vectors. We write $\mathsf{Var}(\eta_j)\triangleq \sigma_j^2$ and let $\mathbb{E}[(\eta_{j})_+^2]=\mathbb{E}[\max\{0,\eta_j\}^2]\triangleq\sigma_{j,+}^2$ denote the positive part of the variance.
\end{enumerate}

\noindent Given these observations and a parameter $\varepsilon>0$, our goal is to return estimates $\widetilde{\w}_1,\ldots,\widetilde{\w}_k\in \mathbb{R}^n$ with the property that $\max_{j\in [k]} \|\widetilde{\w}_j-\w_j\|\leq \varepsilon$. 

With no assumption on the correlations across components of the joint distribution  $\mathcal{D}$, \Cref{eq:lin_reg_ss} simultaneously captures the recent model of \citet{DBLP:conf/stoc/CherapanamjeriD23}, where it is assumed that $\bm{\eta}\sim \mathcal{N}(0,I_k)$, as well as the problem of \emph{max-linear regression} when the addition of noise is taken \emph{outside} the maximum (\citet*{magnani2009convex}). In particular, the i.i.d. samples $(\bm{x}_{\ell},z_{\ell})$ generated by the max-linear regression model are
\begin{equation}
\label{eq:max_lin}
    z_{\ell}=\max_{j\in [k]} \left\{\bm{x}_{\ell}^T\w_j\right\} +\eta_{\ell} \tag{MAX-LIN},
\end{equation}
where $\eta_{\ell}\in \mathbb{R}$ is a centered random variable independent of $\bm{x}_{\ell}$. As observed by \citet{DBLP:conf/stoc/CherapanamjeriD23}, \Cref{eq:max_lin} reduces to  \Cref{eq:lin_reg_ss} by letting $\bm{\eta}\equiv\xi\cdot \bm{1}\in \mathbb{R}^k$ for a centered random variable $\xi\in \mathbb{R}$. While these two previously studied algorithmic problems are quite related, the placement of $\max$ with respect to the noise induces subtle technical differences. In \Cref{eq:max_lin}, the main challenge is computational; minimizing the average square errors of the observations yields meaningful estimates, but the optimization landscape becomes non-convex in the joint tuple of regression parameters $(\theta_1,\ldots,\theta_k)\in \mathbb{R}^{n\times k}$ for $k>1$. However, the index of the maximizer is a deterministic function of $\bm{x}\in \mathbb{R}^n$. On the other hand, when isotropic Gaussian noise is taken inside the maximum as in \citet{DBLP:conf/stoc/CherapanamjeriD23}, even formulating a statistically meaningful optimization problem to recover $\w_1,\ldots,\w_k$ is nontrivial. Since \Cref{eq:lin_reg_ss} captures these settings and beyond, this problem inherits the algorithmic and statistical challenges of both.

To make \Cref{eq:lin_reg_ss} tractable, we further impose the following standard assumptions on the geometry of the true $\w_1,\ldots,\w_k\in \mathbb{R}^n$ and the concentration properties of $\bm{\eta}$. 

\begin{assumption}[$\Delta$-Uncovered]
\label{assumption:uncovered}
    The vectors $\w_1,\ldots,\w_k\in \mathbb{R}^n$ are $\Delta$-uncovered: for some $\Delta>0$, and for each $i\neq j$,
    \begin{equation*}
        \vert\langle \w_j,\w_i\rangle\vert\leq \|\w_i\|^2-\Delta^2.
    \end{equation*}
     In particular, $\|\w_i\|\geq \Delta$ for all $i\in [k]$.
\end{assumption}

\begin{assumption}[$B$-Bounded]
\label{assumption:bounded}
    The vectors $\w_1,\ldots,\w_k\in \mathbb{R}^n$ are $B$-bounded, i.e. $\|\w_i\|\leq B$ for all $i\in [k]$ for some $B\geq 1$. Moreover, the law $\mathcal{D}_j$ of each $\eta_j$ is marginally $B$-subgaussian (see \Cref{def:subgaussian}).
\end{assumption}

These assumptions are both weaker than those of \citet{DBLP:conf/stoc/CherapanamjeriD23}. \Cref{assumption:uncovered} is implied by their analogous uncoveredness condition, though the difference is qualitatively small. \Cref{assumption:bounded} substantially relaxes their distributional assumption that $\bm{\eta}\sim \mathcal{N}(0,I_k)$, and is also assumed for provable guarantees for \Cref{eq:max_lin} (\citet*{ghosh}). We discuss the necessity of some form of subgaussianity for sample-efficiency for any fixed $B,\Delta>0$ in \Cref{rmk:sg_tight}.

\subsection{Our Contributions}
In this paper, we provide the first near optimally sample-efficient algorithm in terms of $k$ for recovering the underlying linear models $\w_1,\ldots,\w_k$ under the general model of linear regression with self-selection bias specified by \Cref{eq:lin_reg_ss}. Moreover, our algorithm succeeds without any \emph{a priori} knowledge of the correlation structure of $\mathcal{D}$. Our main theorem is informally stated as follows:

\begin{thm}[\Cref{thm:final_statement}, informal]
\label{thm:main_intro}
    Under \Cref{assumption:uncovered} and \Cref{assumption:bounded} with fixed $B,\Delta>0$, for any $\varepsilon<\mathsf{poly}(\Delta,1/B)$ and $\lambda\in (0,1)$, there exists an algorithm for \Cref{eq:lin_reg_ss} that outputs $\widetilde{\w_1},\ldots,\widetilde{\w_k}\in \mathbb{R}^n$ satisfying
    $
    \max_{i\in [k]} \|\widetilde{\w_i}-\w_i\|\leq \varepsilon$
with probability at least $1-\lambda$. Moreover, the sample complexity is $\tilde{O}(n)\cdot \mathsf{poly}(k,1/\varepsilon,\log(1/\lambda))$, and the time complexity is $\mathsf{poly}(n,k,1/\varepsilon,\log(1/\lambda))+O\left(\frac{\log(k)}{\varepsilon}\right)^{O(k)}\cdot \log(1/\lambda)$.
\end{thm}
In \Cref{thm:main_intro}, we consider the geometric parameters $B,\Delta>0$ of \Cref{assumption:uncovered} and \Cref{assumption:bounded} to be fixed, so the exponent for $k$ depends polynomially on $B,1/\Delta$. However, for all $B,\Delta>0$, the dependence on $k$ we obtain in the sample complexity is optimal up to a multiplicative constant in this exponent, as it will essentially match the inverse probability that a given regressor attains the maximum in \Cref{eq:lin_reg_ss} under \Cref{assumption:uncovered} and \Cref{assumption:bounded} (see \Cref{rmk:k_tight}). We provide the precise dependencies on all parameters in \Cref{thm:final_statement}.

When restricted to the special case of $\bm{\eta}\sim \mathcal{N}(0,I_k)$, we obtain an exponential improvement in the sample complexity and run time compared to the algorithm of \citet{DBLP:conf/stoc/CherapanamjeriD23} in terms of $k$ and $1/\varepsilon$. The guarantees of their algorithm yield sample complexity $\widetilde{O}(n)\cdot \exp(\mathsf{poly}(B/\varepsilon)\cdot \widetilde{O}(k))$ and time complexity $\mathsf{poly}(n)\cdot \exp(\mathsf{poly}(B/\varepsilon)\cdot \widetilde{O}(k))$ to recover a $\varepsilon$-approximation to $\w_1,\ldots,\w_k$. \Cref{thm:main_intro} provides a significant improvement along both dimensions despite the fact our algorithm solves a more general problem, and thus cannot leverage any \emph{a priori} information about the law of $\bm{\eta}$, unlike the algorithm of \citet{DBLP:conf/stoc/CherapanamjeriD23}. They also provide an alternative, moment-based algorithm that is fully efficient in $n,1/\varepsilon,B,1/\Delta$ for the special case of $k=2$; an immediate corollary of \Cref{thm:main_intro} is a generalization of this for subgaussian noise distributions that holds for any $k=O(1)$.

\begin{cor}
    Under the conditions of \Cref{thm:main_intro}, if $k=O(1)$ is fixed, there exists an algorithm for \Cref{eq:lin_reg_ss} that recovers $\varepsilon$-approximate regressors with probability at least $1-\lambda$ with runtime $\mathsf{poly}(n,1/\varepsilon,\log(1/\lambda))$ and sample complexity $\tilde{O}(n)\cdot \mathsf{poly}(1/\varepsilon,\log(1/\lambda)).$
\end{cor}

In \Cref{rmk:k_tight}, we explain why the case of $k=2$ considered by \citet{DBLP:conf/stoc/CherapanamjeriD23} is somewhat special compared to $k\gg 2$ in terms of the geometry of \Cref{eq:lin_reg_ss}. In particular, an exponential dependence on $B,1/\Delta$ (rather than polynomial as in the result of \citet{DBLP:conf/stoc/CherapanamjeriD23}) for $k\gg 2$ is necessary since even the minimum observation probability of a regressor scales exponentially in these quantities.

Finally, when restricted to the setting of max-linear regression as in \Cref{eq:max_lin}, we obtain an exponential improvement in the time complexity dependence in $k$. The best end-to-end guarantee obtained by \citet*[Corollary 1]{ghosh} under \Cref{assumption:uncovered} and \Cref{assumption:bounded}, which also requires the Gaussian covariate assumption, has time complexity $O(k)^{O(k^2)}\cdot \mathsf{poly}(n,1/\varepsilon)$, and thus is inefficient when $\varepsilon=\Omega(1)$ for much smaller values of $k$ than \Cref{thm:main_intro}. While the guarantee of \Cref{thm:main_intro} is worse when $\varepsilon>0$ is very small, we remark that our algorithm can in general be used as a \emph{warm start} for the local convergence algorithms of \citet{ghosh,kim_lee} to obtain exponentially better $k$ dependence for \emph{any} value of $\varepsilon>0$. Indeed, \citet{ghosh,kim_lee} show that starting with estimates $\widetilde{\w_1},\ldots,\widetilde{\w_k}$ that lie within $\mathsf{poly}(1/k)$ of $\w_1,\ldots,\w_k$ under \Cref{assumption:uncovered} and \Cref{assumption:bounded} with $B,\Delta>0$ fixed, local algorithms like alternating minimization or gradient descent yield rapid convergence for any $\varepsilon>0$. In particular, by directly applying the guarantee of \Cref{thm:main_intro} with $\varepsilon'=\mathsf{poly}(1/k)$ to obtain a $\mathsf{poly}(1/k)$-close warm start for the local convergence algorithm of \citet[Theorem 1]{ghosh}, we obtain:

\begin{cor}
\label{cor:better_max_lin}
    Under \Cref{assumption:uncovered} and \Cref{assumption:bounded} with $B,\Delta>0$ fixed, there exists an algorithm for \Cref{eq:max_lin} that recovers $\varepsilon$-approximate regressors for any $\varepsilon>0$ with probability at least $.99$ with sample complexity $\tilde{O}(n)\cdot \mathsf{poly}(k,1/\varepsilon)$ and time complexity $\mathsf{poly}(n,k,1/\varepsilon)+O(k)^{O(k)}$.
\end{cor}

We remark that the guarantee of \Cref{thm:main_intro} is \emph{agnostic} to the precise correlation structure of $\mathcal{D}$ and can thus be applied without any advance knowledge of the precise \Cref{eq:lin_reg_ss} setting beyond \Cref{assumption:uncovered} and \Cref{assumption:bounded}. However, the local convergence guarantees of \citet{ghosh,kim_lee} are only known to hold for the special case of \Cref{eq:max_lin}, so \emph{a priori} knowledge of this correlation structure is necessary to obtain the guarantees of \Cref{cor:better_max_lin}.

\subsection{Related Work}
To the best of our knowledge, the model \Cref{eq:lin_reg_ss} where noise is taken inside the $\max$ was first formally introduced by \citet{DBLP:conf/stoc/CherapanamjeriD23} in the case where the noise is isotropic Gaussian. We provide a more detailed discussion of their approach for this special case in \Cref{sec:overview}, as well as various challenges that arise when attempting to generalize and improve the complexity of their algorithm. Their work also considered the easier \emph{known-index} setting where the identity of the maximizing index is observed in each sample. The known-index setting had previously been considered extensively for a variety of applications in the economics and econometrics literature; we refer to the excellent introduction of \citet{DBLP:conf/stoc/CherapanamjeriD23} for several examples and \citet*{lee2001self} for existing econometric methods for this setting. Under natural conditions, \citet{DBLP:conf/stoc/CherapanamjeriD23} provide efficient algorithms in $n,k,1/\varepsilon$ to produce estimates $\widetilde{\w}_1,\ldots,\widetilde{\w}_k$ satisfying $\max_{i\in [k]} \|\widetilde{\w}_i-\w_i\|\leq \varepsilon$ for a wide range of ``convex-selection rules.''\\

\noindent\textbf{Max-Affine Regression.} As described above, the model we consider also captures the special case of max-linear regression. The slightly more general problem where the linear models also have nonzero constant terms has been recently studied under the name \emph{max-affine regression.} This model is commonly motivated as a restriction of convex regression, since the maximum of $k$ linear functions yields a convex mapping, and also generalizes the important problem of \emph{phase retrieval}. While several earlier works considered scalable, heuristic optimization methods towards obtaining (not necessarily optimal) solutions to \Cref{eq:max_lin} (\citet*{magnani2009convex,JMLR:v14:hannah13a,balazs2016convex}), to our knowledge, the only end-to-end provable guarantees were obtained by \citet*{ghosh} also in the case of Gaussian covariates. As with \citet{DBLP:conf/stoc/CherapanamjeriD23}, their work uses spectral methods to localize to a $k$-dimensional subspace and brute forces to find $\mathsf{poly}(1/k)$-approximate solutions. Their work then shows how to rapidly boost this guarantee to arbitrary $\varepsilon$ accuracy using local convergence guarantees for alternating minimization (see also \citet*{kim_lee} for an alternate approach using first-order methods). However, this spectral initialization already suffers from runtime $k^{O(k^2)}\cdot \mathsf{poly}(n,1/\varepsilon)$; the reason is that their brute force search must iterate over $k$-tuples of regressors on a net of size $k^{O(k)}$, whereas our approach performs this step far more efficiently. Directly combining our algorithm for $\varepsilon'=1/\mathsf{poly}(k)$ with their local guarantee leads to Corollary~\ref{cor:better_max_lin}. The recent work of \citet*{max_lin_cp_kim} provides a convex program to solve \Cref{eq:max_lin} with nontrivial error guarantees for arbitrary noise; however, their algorithm is not consistent in the case of stochastic noise.\\

\noindent\textbf{Mixture of Linear Regressions.} 
The problem \Cref{eq:lin_reg_ss} is also superficially related to the problem of learning \emph{mixtures of linear regressions}. In this problem, each observation $(\bm{x}_{\ell},z_{\ell})$ is such that $z_{\ell}$ is \emph{randomly and exogenously} obtained by sampling a fixed distribution over linear models to apply to $\bm{x}$, where this sampling is independent of $\bm{x}$; in \Cref{eq:lin_reg_ss}, the selection is done endogenously depending on $\bm{x}$ via the $\max$ selection rule. This problem has been extensively studied in recent years (\citet*{DBLP:conf/aistats/KwonC20,DBLP:conf/stoc/Chen0S20,DBLP:conf/focs/DiakonikolasK20,DBLP:conf/icml/PalMSG22}). The current state-of-the-art is the algorithm of \cite[Theorem 7]{DBLP:conf/focs/DiakonikolasK20} that runs in time $\mathsf{poly}(n,k^{\log^2 k},1/\varepsilon)$ under natural separation assumptions. The problem of learning mixtures of linear regressions is generalized by \emph{list-decodable regression}; here, it is assumed that small $\alpha>0$ fraction of inlier data $(\bm{x}_{\ell},z_{\ell})\in \mathbb{R}^n\times \mathbb{R}$ is generated by a true linear model (possibly with noise), while the remaining $1-\alpha$ fraction of the data is generated arbitrarily. In this case, one returns a list of small size $f(\alpha)\gg 1$ such that the inlier regression coefficients lie in the list. The best algorithmic guarantees were obtained via the Sum-of-Squares hierarchy independently by \citet*{DBLP:conf/soda/RaghavendraY20} and \citet*{DBLP:conf/nips/KarmalkarKK19} under somewhat stringent (but provably necessary) restrictions on the distribution of the inlier data and noise tolerance; these guarantees are qualitatively matched by recent statistical query lower bounds of \citet*{DBLP:conf/nips/DiakonikolasKPP21}. However, we are not aware of a formal connection between \Cref{eq:lin_reg_ss} and list-decodable regression.\\

\noindent\textbf{Junta Testing.} A related line of work that aims to efficiently uncover latent, low-dimensional linear structure from data is \emph{junta testing}. Recent work by \citet*{DBLP:conf/colt/DeMN19} shows that given query access to an unknown Boolean function $f:\mathbb{R}^n\to \{0,1\}$, one can test whether $f$ is far, under the Gaussian measure on $\mathbb{R}^n$, from any function $h(\bm{x})=g(\w_1^T\bm{x},\ldots,\w_k^T\bm{x})$ with $g:\mathbb{R}^k\to \{0,1\}$ (a \emph{linear $k$-junta}). Their main result provides an algorithm that distinguishes between $f$ being a linear $k$-junta with Gaussian surface area $s$ and $f$ being $\varepsilon$-far from any linear $k$-junta with Gaussian surface area $s(1+\varepsilon)$ using a number of queries that depends only on $k,s,\varepsilon$ (notably with no dependence on $n$). Their algorithm uses Hermite analysis to simulate gradients and eventually approximately compute the rank of a certain Gram matrix. This analysis was later generalized for arbitrary correlation parameters by \citet*{DBLP:conf/focs/DeMN19}, with query complexity again depending $k,s,\varepsilon$. Related ideas were recently used by \citet*{multi_index_dk} for agnostic learning of multi-index models using query access. In all these works, query access makes determining low-dimensional structure significantly easier than sample access due to the ability to efficiently simulate derivatives.\\

\noindent\textbf{Notation and Conventions.} In general, we use bold characters to denote vectors. Given a subspace $V\subseteq \mathbb{R}^n$, we write $P_V(\bm{x})$ to denote the orthogonal projection of $\bm{x}$ onto $V$; in a slight abuse of notation, given a vector $\bm{v}\neq \bm{0}$, we will also write $P_{\bm{v}}$ to denote $P_{\text{span}(\bm{v})}$. We write $V^{\perp}$ to denote the orthogonal complement of $V$. Given a vector $\bm{v}\in \mathbb{R}^n$, we write $\|\bm{v}\|$ to denote the standard Euclidean norm. For any nonzero vector $\bm{v}$, we will write $\overline{\bm{v}}=\bm{v}/\|\bm{v}\|$ to denote the unit vector in the direction of $\bm{v}$. For any $a\in \mathbb{R}$, we write $a_+=\max\{0,a\}$ for the positive part of $a$ and analogously for $a_-$ for the negative part. We denote the $n-1$ dimensional unit sphere by $\mathcal{S}^{n-1}\subseteq \mathbb{R}^n.$

\section{Overview of Techniques}
\label{sec:overview}

\textbf{Prior Work and Challenges.}
Before describing our approach and our techniques in more detail, we briefly explain the high-level approach of \citet*{DBLP:conf/stoc/CherapanamjeriD23}. Their algorithm is based on a careful finite-sample implementation of the following identifiability argument: in the case where the noise is known to be isotropic Gaussian, if the Gaussian covariates $\bm{x}$ are conditioned to lie in the span of a unit vector $\bm{v}$, then the $L_p$-norm (in the sense of random variables) of the observations with appropriate normalization will effectively converge as $p\to \infty$ to the maximum absolute overlap between $\bm{v}$ and a regressor $\bm{w}_j$. Thus, one can maximize this function on $\mathcal{S}^{n-1}$ (up to a certain sign subtlety) to find a largest norm regressor, and then one can find the remaining regressors inductively after debiasing new observations using the known regressor. 

To convert this information-theoretic argument into a finite-sample algorithm, their approach is to first identify a $k$-dimensional subspace $V$ that approximately contains each of the $\w_1,\ldots,\w_k$ by taking the span of top eigenvectors of a certain weighted covariance matrix. 
On the low-dimensional space $V$, they then estimate large moments of the observations conditioned on the covariates having very small projection onto $\mathsf{span}(\bm{v})^{\perp}$ for each $\bm{v}$ in a suitable net of $V$; the lower dimensionality ensures that one pays exponentially in $k$ for the net size rather than the ambient dimension $n$. Their analysis then shows that the local maximizers of these conditional moments essentially cluster near the \emph{directions} of the true vectors $\w_1,\ldots,\w_k$, which can then be used with explicit Gaussian moment calculations to recover the approximate regressors with the right magnitudes. Their final algorithm has sample complexity $n\cdot \exp(\mathsf{poly}(B/\varepsilon)\cdot \widetilde{O}(k))$ and time complexity $\mathsf{poly}(n)\cdot \exp(\mathsf{poly}(B/\varepsilon)\cdot \widetilde{O}(k))$ to recover a $\varepsilon$-approximation to $\w_1,\ldots,\w_k$.

However, this approach appears to necessarily induce several unavoidable dependencies and conditions that prevent quantitatively improved  guarantees in more general settings: 

\begin{itemize}
    \item First, the identifiability argument, and thus its algorithmic analogue, appears specialized to the setting of isotropic Gaussian noise; this occurs due to the need to reason quite exactly about the behavior of high moments, which is explicit in the case of Gaussian random variables. 
    
    \item The sample complexity of recovering the approximate subspace $V$ scales inversely with a certain eigenvalue gap that is lower bounded by the minimal probability that any regressor $\w_i$ is indeed the true maximizer. Their results only establish a $\exp(-\Omega(k\log(B/\Delta)))$ lower bound for this quantity under \Cref{assumption:uncovered} and \Cref{assumption:bounded}.
    \item The moments that their algorithm computes are taken conditional on the event that a standard Gaussian vector has small constant norm in a fixed subspace of dimension $k-1$.  Since the norm of a standard Gaussian has subgaussian concentration about $\sqrt{k}$, this event provably only occurs with $\exp(-\Omega(k))$ probability. 
    \item Finally, the conditional moments needed to carry out the analysis are prohibitively high degree due to the lack of concentration. The requisite power scales at least like $k\cdot \mathsf{poly}(B\log(k)/\varepsilon)$, and so obtaining accurate empirical estimates with high probability for even a single vector will require $\exp\left(k\cdot \mathsf{poly}(B\log(k)/\varepsilon)\right)$ samples. 
\end{itemize}  

It would therefore appear that each of these components necessarily leads to sub-optimal sample complexity, necessitating an alternative algorithmic approach.\\

\noindent\textbf{Our High-Level Approach.} We now give a high-level overview of our overall strategy, deferring the technical components to the next sections. Rather than viewing true regressors as local maximizers of high moments on exponentially unlikely events, which may be difficult to analyze for unknown noise distributions, our approach will instead view them as directions that `jointly explain' \emph{low-degree moments conditional on the much more likely events that $\bm{x}$ is large in given one-dimensional subspaces}. We show that this simple alternative approach already remedies several of the exponential dependencies from \citet{DBLP:conf/stoc/CherapanamjeriD23}.

To do this, we combine a fairly subtle geometric analysis in conjunction with the subgaussian properties of the data to show how a delicate combination of \emph{low-degree} moments conditioned on events of this type can effectively distinguish vectors that are close to a true $\w_i$ from those that are $\varepsilon$-far from all such regressors. When $\bm{v}$ is close enough to a true regressor $\w_i$, a simple debiasing of these conditional moments can be shown to have consistent low-order moments with the \emph{standard linear regression} model induced by $\w_i$ in isolation up to a small error. In the case that $\bm{v}$ is $\varepsilon$-far from any of the $\w_i$, we will have to work somewhat harder to reason about the possible low-degree moment behaviors of these statistics (see the next section for subtle behavior that can arise). We do so by arguing carefully about the possible geometry of the $\w_1,\ldots,\w_k$ in relation to a candidate $\bm{v}$ and the kinds of conditional moments it can have.

Finally, using these structural results, we develop our final algorithm to find $\varepsilon$-accurate estimates to the $\w_1,\ldots,\w_k$. By the results of \citet{DBLP:conf/stoc/CherapanamjeriD23} with our improved probability estimates, we first identify a low-dimensional subspace $V$ approximately containing the $\w_1,\ldots,\w_k$ using $n\cdot \mathsf{poly}(k,1/\varepsilon)$ samples. Our final algorithm proceeds by using the structural results we obtained on the low-order conditional moments to carry out an iterative procedure that picks out a $\varepsilon$-close vector to a true $\w_i$ in each round; this step requires carefully pruning the remaining set of candidate vectors to ensure we will find undiscovered regressors while not selecting near-duplicate vectors to those we have already found. Because our analysis is framed in terms of low-order conditional moments on events with decent probability, we have sufficient concentration to obtain accurate moment estimates simultaneously over all vectors in the net using a near optimal number of samples as a function of $k$ for given $B,\Delta>0$.\\

\noindent\textbf{The Geometry and Moments of Regressors.}
In this section, we explain our approach to using low-degree moments to efficiently distinguish between vectors that are $\gamma$-close to a true $\w_i$ for $\gamma\ll \varepsilon$ and those that are $\varepsilon$-far from all regressors. The quantitative guarantees we obtain directly translate to our final iterative algorithm to extract a $\varepsilon$-close vector to each $\w_i$. 

To improve the sample complexity bounds of \citet{DBLP:conf/stoc/CherapanamjeriD23}, a necessary first step is to show that the minimum observation probability of each of the regressors in \Cref{eq:lin_reg_ss} is not exponentially small in $k$. In fact, we show that under \Cref{assumption:uncovered} and \Cref{assumption:bounded}, the minimum observation probability is lower bounded by $1/\mathsf{poly}(k)$, where the precise exponent depends on $B,\Delta$. We do so by proving that when the data $\bm{x}\sim\mathcal{N}(0,I)$ is conditioned to be only somewhat correlated with $\w_i$, the $i$th regressor will be the maximizer in \Cref{eq:lin_reg_ss} with high probability over all the other (independent) randomness. This argument easily extends to vectors $\bm{v}$ that are suitably close to a $\w_i$; the formal statement is as follows.

\begin{prop}[\Cref{prop:close_maximal}, restated]
\label{prop:close_maximal_intro}
    Suppose that \Cref{assumption:uncovered} and \Cref{assumption:bounded} hold and let $\bm{0}\neq \bm{v}\in \mathbb{R}^n$. For any $t\geq CB^2/\Delta^2$ for an absolute constant $C>0$, and any $\delta\in (0,1)$, define the event 
    \begin{equation*}
    A_{t,\overline{\bm{v}},\delta}\triangleq\left\{t\sqrt{\log(k/\delta)}\leq \bm{x}^T\bm{\overline{v}}\leq 2t\sqrt{\log(k/\delta)}\right\}.
    \end{equation*}
    Then $\Pr_{\bm{x}\sim \mathcal{N}(0,I)}\left(A_{t,\overline{\bm{v}},\delta}\right)= (\delta/k)^{\Theta(t^2)}$ and
 moreover, if there exists $i\in [k]$ such that $\|\bm{v}-\w_i\|\leq \mathsf{poly}(\Delta,1/B)$, then $
            \Pr\left(i\neq\arg\max_{j\in [k]} \bm{x}^T\w_j+\eta_j\bigg\vert A_{t,\overline{\bm{v}},\delta}\right)\leq \delta.$
\end{prop}

The intuition is straightforward: \Cref{assumption:uncovered} and \Cref{assumption:bounded} can be shown to imply a quantitative separation between $\w_i$ and each of the $\w_j$ when written in a $k$-dimensional orthonormal basis where the first coordinate corresponds to the direction $\overline{\bm{v}}$. Thus, the $i$th regressor will be maximal so long as the projection of $\bm{x}$ onto $\overline{\bm{v}}$ sufficiently exceeds the maximum of a certain subgaussian process of cardinality $k$ with marginal subgaussian norm $O(B)$. By well-known tail bounds on the maxima of such processes, it will follow that the $i$th regressor will be maximal on $A_{t,\overline{\bm{v}},\delta}$ with probability at least $1-\delta$ over the remaining, \emph{independent} randomness. The bounds on $\Pr(A_{t,\overline{\bm{v}},\delta})$ are a straightforward consequence of tight Gaussian tail bounds. These observability bounds can easily be shown to be quantitatively tight; we provide an explicit example in \Cref{rmk:k_tight}.

While \Cref{prop:close_maximal_intro} certifies a stronger probability lower bound by conditioning on an exponentially more likely event than that of \citet{DBLP:conf/stoc/CherapanamjeriD23}, for this to be useful, we will have to carry out the rest of our analysis only on weaker events of this type. In the case that $\bm{v}\approx \w_i$ for some $i\in [k]$, the quantitative guarantee of \Cref{prop:close_maximal_intro} in tandem with subgaussianity of the samples imply strong low-order moment bounds when localized to $\bm{v}$. Recall $\mathsf{Var}(\eta_j)\triangleq \sigma_j^2$ and  $\mathbb{E}[(\eta_{j})_+^2]\triangleq\sigma_{j,+}^2$ denotes the positive part of the variance.

\begin{cor}[\Cref{cor:conditional_moments}, informal]
\label{cor:conditional_moments_intro}
    Under \Cref{assumption:uncovered} and \Cref{assumption:bounded}, the following holds for any $\gamma>0$ and any $t\geq CB^2/\Delta^2$. Suppose $\|\bm{v}-\w_i\|\leq \min\{\mathsf{poly}(\Delta,1/B),\mathsf{poly}(\gamma,1/t,1/\log(k))\}$, and let
   $\delta=\mathsf{poly}(\gamma,1/\log(k),1/B)$.
Then on the event $A_{t,\overline{\bm{v}},\delta}$, the random variable $
        Y \triangleq z - \bm{v}^T\bm{x}$
    satisfies \begin{equation*}
        \vert \mathbb{E}[Y\vert A_{t,\overline{\bm{v}},\delta}] \vert\leq \gamma,\quad\quad \vert\mathbb{E}[Y^2\vert A_{t,\overline{\bm{v}},\delta}] - \sigma_i^2\vert\leq \gamma,\quad\quad \vert\mathbb{E}[Y_+^2\vert A_{t,\overline{\bm{v}},\delta}] - \sigma_{i,+}^2\vert\leq \gamma.
    \end{equation*}
\end{cor}

In words, \Cref{cor:conditional_moments_intro} implies that when the samples are localized to a direction close enough to $\w_i$, the low-order moments of the debiased random variable $Y$ closely approximate the \emph{standard linear regression moments} (with $k=1$) with regressor $\w_i$. However, unlike the isotropic Gaussian setting of \citet{DBLP:conf/stoc/CherapanamjeriD23}, the ``correct'' conditional moments remain unknown under our weaker \Cref{assumption:uncovered}, so one cannot directly exploit the precise values of these moments even when $\bm{v}=\w_i$. Thus, we require a more subtle comparison between this kind of moment behavior and that of $\varepsilon$-far vectors $\bm{v}$ from every $\w_i$. 

Indeed, it will not even be the case that \emph{any} of the $\w_i$ can be interpreted as minimizers of low-order moments on these conditional events. It is not difficult to construct examples of $\w_1,\ldots,\w_k$ and a far vector $\bm{v}$ such that, conditional on $\bm{x}^T\bm{v}$ being relatively large, $\bm{v}$ `explains' the low-order moments of the observations. To understand how this is possible, let $\bm{v}=1.2\bm{e}_{k+1}$ and let $\w_i=\bm{e}_i+\bm{e}_{k+1}$ with $\bm{\eta}\sim \mathcal{N}(0,I_k)$. Suppose we conditioned on the event $x_{k+1}\sim \mathcal{N}(0,1)=10\sqrt{\log(k)}$ exactly; while the events $A_{t,\overline{\bm{v}},\delta}$ use a nontrivial interval, this approximation will hold up to a $o(1)$ additive term since $\bm{x}^T\bm{v}$ will concentrate on the lower endpoint. On this event,
\begin{equation*}
    \max_{i\in [k]}\{\bm{x}^T\w_i+\eta_i\}=10\sqrt{\log(k)}+\max_{i\in [k]}\{x_i+\eta_i\}.
\end{equation*}
The latter maximum, which is independent of the conditioning by the orthogonality of the Gaussian directions, concentrates about $2\sqrt{\log(k)}$ with $o_k(1)$ fluctuations (by e.g. Talagrand's $L_1$-$L_2$ inequality, see \cite[Example 8.20]{vanhandel}). Thus, the random variable $z-\bm{x}^T\bm{v}$ can have $o(1)$ low order moments when conditioning on $\bm{x}^T\overline{\bm{v}}\approx t\sqrt{\log(k)}$ for some particular $t>0$. From the perspective of the low-order moments, such a $\bm{v}$ seems to nearly \emph{perfectly} explain these conditional observations as if it were a regressor without noise, even better than any of the actual $\w_i$!   

However, by considering the possible configurations of far vectors in relation to the $\w_i$, we will show that the first and second moments at \emph{just two} thresholds $t_1,t_2=\Theta(B^2\sqrt{\log(k/\delta)}/\Delta^2)$, for suitable $\delta>0$ will be inconsistent with being a true regressor in a somewhat subtle sense that we describe shortly. We prove the following geometric result:

\begin{prop}[\Cref{lem:no_large_proj,lem:no_two_close,lem:not_all_neg}, informal]
\label{prop:unique_close_intro}
    Let $\delta>0$ and $\gamma\ll \varepsilon$ and suppose $\bm{0}\neq \bm{v}\in \mathbb{R}^n$ is such that $\min_{j\in [k]} \|\bm{v}-\w_j\|\geq \varepsilon$. Suppose that 
    \begin{equation*}
        \max\left\{\left\vert \mathbb{E}[z-\bm{v}^T\bm{x}\vert A_{t,\overline{\bm{v}},\delta}]\right\vert,\left\vert\mathbb{E}[z-\bm{v}^T\bm{x}\vert A_{4t,\overline{\bm{v}},\delta}]\right\vert\right\} \leq \gamma.
    \end{equation*}
    Then there must exist a unique $i\in [k]$ such that
    \begin{equation*}
        \|P_{\overline{\bm{v}}}(\w_i)-\bm{v}\|\lesssim \frac{\gamma}{t\sqrt{\log(k/\delta)}}.
    \end{equation*}
\end{prop}

In words, \Cref{prop:unique_close_intro} shows that if $\bm{v}$ passes two simple first moment tests conditional on events of the form $\bm{v}^T\bm{x}$ is large, this is explained by the existence of a unique $\w_i$ whose projection onto $\mathsf{span}(\bm{v})$ lies quite close to $\bm{v}$. The proof of \Cref{prop:unique_close_intro} applies the following reasoning: first, if the projection of any $\w_i$ onto $\mathsf{span}(\bm{v})$ is noticeably larger than $\bm{v}$, then the first moment of $z-\bm{v}^T\bm{x}$ conditioned on  $A_{t,\overline{\bm{v}},\delta}$ will be noticeably large than zero by stochastic domination.

On the other hand, if there exists distinct $i\neq j$ such that $P_{\overline{\bm{v}}}(\w_i)\approx P_{\overline{\bm{v}}}(\w_j)\approx \overline{\bm{v}}$, then another stochastic domination argument will imply that the first moment of $z-\bm{v}^T\bm{x}$ on $A_{t,\overline{\bm{v}},\delta}$ will exceed the maximum of two noticeable Gaussians that are obtained by projecting to $\mathsf{span}(\bm{v})^{\perp}$. At this point, we will be able to carefully leverage the geometry of $\w_i$ and $\w_j$ under \Cref{assumption:uncovered} and \Cref{assumption:bounded} to argue that either the remaining Gaussian components after projecting to $\mathsf{span}(\bm{v})^{\perp}$, which are independent of the event $A_{t,\overline{\bm{v}},\delta}$, remain somewhat independent or negatively correlated. In either case, the maximum of these Gaussian random variables will be noticeably larger than $0$ as needed.

Finally, if all $P_{\overline{\bm{v}}}(\w_i)$ are noticeably \emph{below} $\bm{v}$ in that subspace, then using \emph{two} thresholds $t_1,t_2$ will ensure that the first moments conditioned on $A_{t_1,\overline{\bm{v}},\delta}$ and $A_{t_2,\overline{\bm{v}},\delta}$ will be noticeably far apart by a simple coupling argument, and thus at least one of them must be noticeably far from $0$. The use of two thresholds is precisely what rules out pathological examples as above where the $\w_1,\ldots,\w_k$ can conspire to cause superconcentration of $z-\bm{v}^T\bm{x}$ on conditional events.

To distinguish $\varepsilon$-far vectors from truly close vectors to a $\w_i$, we then prove the following result showing that a $\varepsilon$-far vector with a $\w_i$ that projects very close to it must have an inflated truncated second moment compared to $\w_i$:

\begin{prop}[\Cref{cor:inflated_sm}, informal]
\label{prop:inflated_sm_intro}
    There is an absolute constant $c>0$ such that the following holds. Suppose that $\min_{j\in [k]} \|\bm{v}-\w_j\|\geq \varepsilon$, but there exists $i\in [k]$ such that
    \begin{equation*}
        \| \bm{v}-P_{\overline{\bm{v}}}(\w_i)\|\leq \frac{c\varepsilon^3}{2B^2\log(B/c\varepsilon)t\sqrt{\log(k/\delta)}}.
    \end{equation*}
    Then it holds that
    \begin{equation*}
     \mathbb{E}[(z-\bm{v}^T\bm{x})_+^2\vert A_{t,\overline{\bm{v}},\delta}]-\sigma_{i,+}^2\geq \frac{c\varepsilon^4}{B^2\log(B/c\varepsilon)}.
    \end{equation*}
\end{prop}

While comparing the variances may appear more natural, the truncated square is nondecreasing and convex, and hence one can again apply stochastic domination arguments again. After debiasing by $\bm{v}^T\bm{x}$ up to a small additive term, the remainder is the sum of $\eta_i$ with an independent Gaussian of variance $\Omega(\varepsilon^2)$ using the farness assumption and the condition that $\bm{v}\approx P_{\overline{\bm{v}}}(\w_i)$. We show via elementary lower bounds that the addition of Gaussian noise must noticeably inflate the truncated second moment just under the subgaussianity assumption on $\eta_i$. While this would be straightforward if $\bm{\eta}$ was assumed isotropic Gaussian by stability of sums of independent Gaussians, our more general analysis requires lower bounds on the probability that a centered, $B$-subgaussian random variable is nonnegative. We provide a tight, elementary bound on this quantity (\Cref{lem:sg_to_lbs}) in terms of the variance and subgaussian norm which may be of independent interest.

Taken together, Proposition~\ref{prop:close_maximal_intro}, Corollary~\ref{cor:conditional_moments_intro}, Proposition~\ref{prop:unique_close_intro}, and Proposition~\ref{prop:inflated_sm_intro} establish the following key characterization of the distinction between the low-order moments of very close vectors and $\varepsilon$-far vectors: a very close vector $\bm{u}$ to any $\w_i$ will have very close low-order conditional moments to $\w_i$, while any $\varepsilon$-far vector from all regressors will either have noticeably large mean at one of two thresholds, or noticeably higher truncated second moment than a unique $\w_j$ that nearly projects onto it.\\

\noindent\textbf{Finding a Low-Dimensional Subspace.}
In order to apply the structural results of the preceding section, we must first reduce the dimensionality of searching for true regressors using the argument of \citet{DBLP:conf/stoc/CherapanamjeriD23}. Roughly speaking, they show that the span of the top $k$ eigenvectors of a certain weighted covariance matrix approximately contains $\mathsf{span}(\w_1,\ldots,\w_k)$. For the convenience of the reader, we provide a slightly more complete and self-contained proof of the relevant matrix concentration in \Cref{sec:subspace_appendix}, while relying on their analysis for a \emph{spectral lower bound} on this population covariance matrix using Gaussian integration by parts on the covariates. The only difference is that the sample complexity of obtaining a precise enough subspace depends inversely on an eigengap that is lower bounded by the minimum probability that a regressor attains the maximum in \Cref{eq:lin_reg_ss}. While their work only established a weak exponentially small bound on this quantity, \Cref{prop:close_maximal_intro} provides inverse polynomial bounds that ensure $\widetilde{O}(n)\cdot \mathsf{poly}(k,1/\varepsilon)$ sample complexity for this step. As this argument, modulo our improved probability estimates, is due to \citet{DBLP:conf/stoc/CherapanamjeriD23}, we defer further discussion to \Cref{sec:subspace_appendix}.\\

\noindent\textbf{Complete Algorithm and Sample Complexity Bounds.}
We can now describe how we put together the previous analyses to obtain our final algorithm. By the previous section, we may assume access to a $k$-dimensional subspace $V$ with the property that every $\w_i$ satisfies $\|P_{V}(\w_i)-\w_i\|\leq \mathsf{poly}(\varepsilon/\log(k))$, where we hide polynomial dependencies on $B,\Delta$. We may then construct a $\mathsf{poly}(\varepsilon/\log(k))$-net $\mathcal{H}$ of the annulus between $\Delta$ and $B$ in $V$; it suffices to only consider this set since \Cref{assumption:uncovered} and \Cref{assumption:bounded} imply these norm bounds for the true regressors. Therefore, for every $\w_i$, there will exist a vector $\widehat{\w}_i\in \mathcal{H}$ such that $\|\widehat{\w}_i-\w_i\|\leq \mathsf{poly}(\varepsilon/\log(k))$. Note that the size of $\mathcal{H}$ is at most $(\log(k)/\varepsilon)^{O(k)}$ by standard bounds on $k$-dimensional nets.

Our algorithm then proceeds as follows: given $m$ i.i.d. samples $\mathcal{S}=(\bm{x}_{\ell},z_{\ell})_{\ell=1}^m$ from the model, a vector $\bm{v}\neq \bm{0}$, and parameters $s,\alpha$, we define the following sets and statistics:
\begin{gather*}
    \mathcal{A}_{s,\overline{\bm{v}},\alpha} = \left\{(\bm{x}_{\ell},z_{\ell})\in \mathcal{S}: s\sqrt{\log(k/\alpha)}\leq \bm{x}_{\ell}^T\overline{\bm{v}}\leq 2s\sqrt{\log(k/\alpha)}\right\}\\
    M^1_{s,\bm{v},\alpha} = \frac{1}{\vert \mathcal{A}_{s,\overline{\bm{v}},\alpha}\vert}\sum_{(\bm{x}_{\ell},z_{\ell})\in \mathcal{A}_{s,\overline{\bm{v}},\alpha}}(z_{\ell}-\bm{v}^T\bm{x}_{\ell})\\
    M^2_{s,\bm{v},\alpha} = \frac{1}{\vert \mathcal{A}_{s,\overline{\bm{v}},\alpha}\vert}\sum_{(\bm{x}_{\ell},z_{\ell})\in \mathcal{A}_{s,\overline{\bm{v}},\alpha}}(z_{\ell}-\bm{v}^T\bm{x}_{\ell})_+^2.
\end{gather*}
In words, $\mathcal{A}_{s,\overline{\bm{v}},\alpha}$ is the set of samples satisfying the event $A_{s,\overline{\bm{v}},\alpha}$, and $M^1_{s,\bm{v},\alpha}$ and $M^2_{s,\bm{v},\alpha}$ are empirical estimates for $\mathbb{E}[z-\bm{x}^T\bm{v}\vert \mathcal{A}_{s,\overline{\bm{v}},\alpha}]$ and $\mathbb{E}[(z-\bm{x}^T\bm{v})_+^2\vert \mathcal{A}_{s,\overline{\bm{v}},\alpha}]$, respectively. 

For each $\bm{v}\in \mathcal{H}$, we compute $M^1_{t,\bm{v},\delta}, M^1_{4t,\bm{v},\delta}$ and $M^2_{t,\bm{v},\delta}$ for $t=CB^2/\Delta^2$ and $\delta=\mathsf{poly}(\varepsilon/\log(k))$ (again hiding polynomial dependencies on $B,1/\Delta$). By \Cref{prop:close_maximal_intro}, and assuming sufficiently accurate empirical estimates for now, we may safely reject any $\bm{v}\in \mathcal{H}$ such that $\max\{\vert M^1_{t,\bm{v},\delta}\vert,\vert M^1_{4t,\bm{v},\delta}\vert\}$ is larger than it should be if $\bm{v}$ were a very close vector to a $\w_i$. By \Cref{prop:unique_close_intro}, we deduce that every remaining vector $\bm{v}$ must be $\varepsilon$-close to a true $\w_i$, or there must exist a unique $i\in [k]$ such that $P_{\overline{\bm{v}}}(\w_i)\approx \bm{v}$. In particular, every remaining vector can be ``assigned'' uniquely to a true $\w_i$ either by proximity or close projection. To make progress, by \Cref{prop:inflated_sm_intro}, we know that every remaining $\bm{v}$ of the latter type has noticeably higher truncated second moment than their assigned $\w_i$. Therefore, we proceed iteratively: at each iteration, we select $\bm{v}^*$ among the remaining vectors with \emph{minimum} truncated second moment. Due to statistical error, we cannot quite deduce that $\bm{v}^*$ is very close to some $\w_i$, but we can deduce that it is $\varepsilon$-close since every remaining $\varepsilon$-far vector is strictly dominated by some other candidate vector with respect to this statistic.

Thus, we certainly find a $\varepsilon$-close vector to some $\w_i$ at the beginning of the first iteration, but we must ensure we make progress in future iterations. Namely, we must avoid selecting other vectors that also correspond to $\w_i$ in future rounds, while also preserving all $\varepsilon$-close vectors to the $\w_j$ that have not been found so far. Here, the utility of \Cref{prop:inflated_sm_intro} becomes even more apparent; a slightly more robust version of this result holds for any $\bm{w}$ that is very close to $\w_i$, like $\widehat{\w}_i$. Unfortunately, $\bm{v}^*$ is only known to be $\varepsilon$-close to $\w_i$, so may not have exceptionally close projection onto the other $\bm{v}$ that are assigned to $\w_i$. However, a consequence of our structural results is that any such vector $\bm{v}$ either is already in the $2\varepsilon$-ball about $\bm{v}^*$ (by the triangle inequality), \emph{or there exists a remaining $\w$ in the $2\varepsilon$-ball about $\bm{v}^*$ (namely, $\widehat{\w}_i$) that must project very close to $\bm{v}$}. After finding $\bm{v}^*$, we can simply collect all remaining points that are within $2\varepsilon$ of $\bm{v}^*$ and delete any other remaining $\bm{v}$ if there exists a point in this ball that projects very close to $\bm{v}$. After also removing the $2\varepsilon$-ball from the remaining points, \Cref{assumption:uncovered} and \Cref{assumption:bounded} can be shown to imply that this procedure does not delete any $\varepsilon$-close points to a different $\w_j$ while deleting the remaining points ``assigned'' to $\w_i$. This completes a sketch of correctness of our algorithm, and the claimed runtime follows easily so long as the estimation of accurate moments is efficient.

Here, the fact that our structural characterization is in terms of \emph{low-degree} moments is pivotal. For a given event $A_{t,\overline{\bm{v}},\delta}$, it can be easily shown that the subexponential norm of the relevant first and second moments is $\mathsf{poly}(B,1/\Delta,\log(k/\varepsilon))$. Therefore, given just $\mathsf{poly}(B,1/\Delta,\log(k),1/\varepsilon,\log(1/\lambda))$ independent samples conditioned on $A_{t,\overline{\bm{v}},\delta}$, one can obtain sufficiently accurate estimators of the first and second moments for a given $\bm{v}\in \mathcal{H}$ with high probability, and moreover, this can be done for all $\bm{v}\in \mathcal{H}$ with minimal overhead due to logarithmic error dependence in the relevant concentration bounds.  Finally, because the events $A_{t,\overline{\bm{v}},\delta}$ have nonnegligible probability (by \Cref{prop:close_maximal_intro}) and low VC-dimension, we can obtain accurate estimates for \emph{all} moments by simply oversampling by a small factor over the number of samples we need for a single $A_{t,\overline{\bm{v}},\delta}$. Putting these bounds together leads to the claimed sample and time complexities in \Cref{thm:main_intro}.

\section{Preliminaries}
\subsection{Linear Algebra}

Given a symmetric matrix $A\in \mathbb{R}^{n\times n}$, we write $\|A\|_{\mathsf{op}}$ to denote the operator norm (equivalently, the largest eigenvalue in magnitude). We will also use the Loewner order notation $A\preceq B$ for symmetric matrices $A,B\in \mathbb{R}^{n\times n}$ to indicate $B-A$ is positive semidefinite.

We observe the following separation consequence of \Cref{assumption:uncovered} and \Cref{assumption:bounded}:

\begin{lem}
\label{lem:sep}
    Under \Cref{assumption:uncovered} and \Cref{assumption:bounded}, for any $i\neq j$, it holds that 
    \begin{equation*}
        \|\w_j-\w_i\|\geq \Delta^2/B.
    \end{equation*}
\end{lem}
\begin{proof}
    We simply calculate 
    \begin{equation*}
        \|\w_j-\w_i\|^2\geq \|P_{\overline{\bm{w}_i}}(\w_j)-\w_i\|^2=(\|\w_i\|-\|P_{\overline{\bm{w}_i}}(\w_j)\|)^2\geq (\Delta^2/\|\w_i\|)^2\geq \Delta^4/B^2.
    \end{equation*}
    by \Cref{assumption:uncovered} and \Cref{assumption:bounded}, using the contraction of projections in the first step.
\end{proof}

We will also require the following simple linear-algebraic facts. The first bounds the difference of orthogonal projections onto one-dimensional subspaces by the distance between the unit vectors in each subspace (note that one may equivalently apply the lemma to $-\bm{u}$ instead of $\bm{u}$, so the right hand side can be replaced by the better of the two).
\begin{lem}
\label{lem:close_proj}
    Let $\bm{u},\bm{v}\in \mathbb{R}^n$ be unit vectors. Then for any vector $\bm{x}\in \mathbb{R}^n$, it holds that
    \begin{equation*}
        \|P_{\bm{u}}(\bm{x})-P_{\bm{v}}(\bm{x})\|\leq 2\|\bm{u}-\bm{v}\|\|\bm{x}\|.
    \end{equation*}
\end{lem}
\begin{proof}
    Recall that $P_{\bm{u}}=\bm{u}\bm{u}^T$ (and analogously for $\bm{v}$) by the assumption that it is a unit vector. It thus suffices to show that $\|P_{\bm{u}}-P_{\bm{v}}\|_{\mathsf{op}}=\|\bm{u}\bm{u}^T-\bm{v}\bm{v}^T\|_{\mathsf{op}}\leq 2\|\bm{u}-\bm{v}\|$. But this is immediate by the triangle inequality:
    \begin{align*}
        \|\bm{u}\bm{u}^T-\bm{v}\bm{v}^T\|_{\mathsf{op}}&=\|\bm{u}\bm{u}^T-\bm{u}\bm{v}^T+\bm{u}\bm{v}^T-\bm{v}\bm{v}^T\|_{\mathsf{op}}\\
        &\leq \|\bm{u}(\bm{u}-\bm{v})^T\|_{\mathsf{op}}+\|(\bm{u}-\bm{v})\bm{v}^T\|_{\mathsf{op}}\\
        &\leq 2\|\bm{u}-\bm{v}\|.\qedhere
    \end{align*}
\end{proof}

We will need the following simple fact about distortion when radially projecting onto the unit sphere:
\begin{lem}
    \label{lem:contractions}
    For any nonzero $\bm{u},\bm{v}\in \mathbb{R}^n$, it holds that
    \begin{equation*}
        \left\|\frac{\bm{u}}{\|\bm{u}\|}-\frac{\bm{v}}{\|\bm{v}\|}\right\|\leq \frac{1}{\min\{\|\bm{u}\|,\|\bm{v}\|\}} \|\bm{u}-\bm{v}\|.
    \end{equation*}
\end{lem}
\begin{proof}
    Without loss of generality, suppose $\|\bm{u}\|\geq \|\bm{v}\|$. Since the orthogonal projection onto any convex set is 1-Lipschitz, we have
    \begin{equation*}
        \|\bm{u}-\bm{v}\|\geq \left\|\frac{\|\bm{v}\|}{\|\bm{u}\|}\bm{u}-\bm{v}\right\|=\|\bm{v}\|\left\|\frac{\bm{u}}{\|\bm{u}\|}-\frac{\bm{v}}{\|\bm{v}\|}\right\|.\qedhere
    \end{equation*}
\end{proof}

\subsection{Subgaussian and Subexponential Random Variables}
We will repeatedly need to appeal to a host of concentration inequalities derived from subgaussianity and subexponentiality. For our purposes, the following will be the most convenient definitions:

\begin{defn}
\label{def:subgaussian}
    The \textbf{subgaussian norm} of a random variable $X$, denoted $\|X\|_{\psi_2}$, is given by
    \begin{equation*}
        \|X\|_{\psi_2} = \inf\{t\geq 0: \mathbb{E}[\exp(X^2/t^2)]\leq 2\}.
    \end{equation*}

    The \textbf{subexponential norm} of a random variable $X$, denoted $\|X\|_{\psi_1}$, is given by
    \begin{equation*}
        \|X\|_{\psi_1} = \inf\{t\geq 0: \mathbb{E}[\exp(\vert X\vert/t)]\leq 2\}.
    \end{equation*}

    We say a random variable $X$ is \textbf{$K$-subgaussian} (repectively, \textbf{$K$-subexponential}) if $\|X\|_{\psi_2}\leq K$ ($\|X\|_{\psi_1}\leq K$).
\end{defn}
It is well-known that $\|\cdot\|_{\psi_1}$ and $\|\cdot\|_{\psi_2}$ are genuine norms and so satisfy the triangle inequality.

\begin{prop}[Proposition 2.5.2 of \citet*{vershynin}]
\label{prop:subgaussian_equivalent}
    There exists an absolute constant $C>0$, such that the following holds: if $X$ is $K$-subgaussian, then for all $t\geq 0$,
    \begin{equation*}
        \Pr(\vert X\vert\geq t)\leq 2\exp(-t^2/CK^2).
    \end{equation*}
    and for all $p\geq 1$,
    \begin{equation*}
        \mathbb{E}[\vert X\vert^p]^{1/p}\leq CK\sqrt{p}.
    \end{equation*}
    In particular, if $X$ is $K$-subgaussian, then $\Var(X)\leq \mathbb{E}[X^2]\leq 2C^2K^2.$
Conversely, $X$ is $C'K$-subgaussian for some other constant $C'>0$ if either of these conditions hold.

\end{prop}

An immediate observation is the following:
\begin{fact}
\label{fact:sg_domination}
    If $X$ and $Y$ are random variables defined on the same probability space such that $\vert X\vert\leq \vert Y\vert$ almost surely, then if $Y$ is $K$-subgaussian (subexponential), so is $X$. In particular, if $X$ is $K$-subgaussian (subexponential), then so is $X_+$.
\end{fact}

We also have the following useful relations between subgaussian and subexponential random variables:

\begin{fact}[Lemma 2.6.8 and Exercise 2.7.10 of \citet*{vershynin}]
\label{fact:subgaussian_additive}
    There exists an absolute constant $C>0$ such that if $X$ is $K$-subgaussian (subexponential), then $X-\mathbb{E}[X]$ is centered and $CK$-subgaussian (subexponential).
\end{fact}

\begin{fact}[Lemma 2.7.7 of \citet*{vershynin}]
\label{lem:sg_to_se}
    If $X$ is $K$-subgaussian, then $X$ is also $2K$-subexponential. Moreover, if $X$ is $K$-subgaussian, then $X^2$ is $K^2$-subexponential and hence $X^2-\mathbb{E}[X^2]$ is $CK^2$-subexponential by \Cref{prop:subgaussian_equivalent} for some absolute constant $C>0$.
\end{fact}

\begin{fact}[Lemma 5.2 of \citet*{vanhandel}]
\label{fact:gaussian_maximal_tail}
     There exists an absolute constant $C>0$ such that the following holds. Suppose that for each $t\in T$, $X_t$ is a $K$-subgaussian random variable where $T$ is a finite index set. Then for all $x\geq 0$,
    \begin{equation*}
        \Pr\left(\sup_{t\in T} X_t\geq CK\sqrt{\log(\vert T\vert)}+x \right)\leq \exp\left(-x^2/CK^2\right).
    \end{equation*}
\end{fact}

\begin{cor}
\label{cor:max_subgaussian}
    There exists a universal constant $C>0$ such that the following holds: suppose that for each $t\in T$, $X_t$ is a $K$-subgaussian random variable where $T$ is a finite index set. Then the random variables $\max_{t\in T} X_t$ and  $\max_{t\in T} X_t - \mathbb{E}[\max_{t\in T} X_t]$ are $CK\sqrt{\log(2\vert T\vert)}$-subgaussian.
\end{cor}
\begin{proof}
    First, note that \Cref{fact:gaussian_maximal_tail} implies that $\mathbb{E}[\max_t X_t]=O(K\sqrt{\log \vert T\vert})$ by integrating the tail. By \Cref{fact:subgaussian_additive}, it thus suffices to prove the claim about the random variable $\max_{t\in T} X_t$. By \Cref{prop:subgaussian_equivalent}, it suffices to show that there is some constant $C'>0$ such that
    \begin{equation*}
        \Pr\left(\sup_{t\in T} \vert X_t\vert\geq x \right)\leq 2\exp\left(-x^2/C'K^2\log(2\vert T\vert)\right).
    \end{equation*}
    If $C'$ is large enough, then for any $x\leq 2CK\sqrt{\log(2\vert T\vert)}$, where $C$ is as in \Cref{fact:gaussian_maximal_tail}, this holds since the right side will be at least $1$ so the probability bound holds trivially.

    If instead $x=CK\sqrt{\log(2\vert T\vert)}+u$ where $u\geq CK\sqrt{\log(2\vert T\vert)}$, then the desired tail bound holds by directly  applying \Cref{fact:gaussian_maximal_tail} with possibly a slightly larger constant $C''>0$ since the denominator is only larger while $u\geq x/2$. The extra factor of $2$ in the logarithm arises by extending the family from $(X_t)_{t\in T}$ to $(X_t)_{t\in T}\cup (-X_t)_{t\in T}$ to control the absolute value.
\end{proof}

Note that if $(X_t)_{t\in T}$ formed a \emph{Gaussian} process, then the logarithmic term is unnecessary by subgaussian concentration of Lipschitz functions of Gaussians. However, paying this extra term will enable us to generalize the model with minimal cost to the sample complexities.

\begin{fact}
\label{fact:decomposition}
    For any random variable $X$ with $\mathbb{E}[X]=0$,  $\mathbb{E}[X_+]=-\mathbb{E}[X_-]=\mathbb{E}[\vert X\vert]/2$.
\end{fact}

We will require the following lemma on lower bounds on the positive parts of centered, $K$-subgaussian random variables to carry out our general analysis:
\begin{lem}
\label{lem:sg_to_lbs}
    There exists an absolute constant $c>0$ such that the following holds: suppose that $X$ is a $K$-subgaussian random variable with zero mean and nonzero variance. Then
    \begin{equation*}
        \mathbb{E}[X_+]\geq \frac{c\Var(X)}{K\sqrt{\log(K^2/c\Var(X))}},
    \end{equation*}
    and
    \begin{equation*}
        \Pr(X\geq 0)\geq \frac{c\Var(X)}{K^2\log(K^2/c\Var(X))}.
    \end{equation*}
\end{lem}
\begin{proof}
    For the first inequality, observe that by \Cref{fact:decomposition}, it suffices to show the same inequality for $\vert X\vert$. By \Cref{prop:subgaussian_equivalent}, there exists an absolute constant $C>0$ such that
    \begin{equation*}
        \mathbb{E}[X^4]^{1/2}\leq CK^2.
    \end{equation*}
    We also observe that for some other constant $C'>0$, we again have by \Cref{prop:subgaussian_equivalent} that
    \begin{equation*}
        \Pr(\vert X\vert\geq C'K\sqrt{\log(C'K^2/\Var(X))})\leq \Var^2(X)/4C^2K^4.
    \end{equation*}
    It follows from Cauchy-Schwarz that
    \begin{equation*}
        \mathbb{E}[X^2\mathbf{1}(\vert X\vert\geq C'K\sqrt{\log(C'K^2/\Var(X))})]\leq \Var(X)/2.
    \end{equation*}

    We deduce that
    \begin{equation*}
        \mathbb{E}[\vert X\vert]\geq \frac{\mathbb{E}[X^2\mathbf{1}(\vert X\vert\leq C'K\sqrt{\log(C'K^2/\Var(X))})]}{C'K\sqrt{\log(C'K^2/\Var(X))}}\geq \frac{\Var(X)}{2C'K\sqrt{\log(C'K^2/\Var(X))}}.
    \end{equation*}

    For the second inequality, observe that by Cauchy-Schwarz again,
    \begin{equation*}
        \mathbb{E}[X_+]\leq \sqrt{\Var(X)\Pr(X\geq 0)},
    \end{equation*}
    so combining with our previous inequality, we deduce
    \begin{equation*}
        \Pr(X\geq 0)\geq \frac{c\Var(X)}{K^2\log(K^2/c\Var(X))}
    \end{equation*}
    for a possibly different constant $c>0$.
\end{proof}

Curiously, both of these bounds are tight including the logarithmic factor. To see this, consider a random variable $X$ such that $X=0$ with probability $1-2/n$ and is symmetric and proportional to $\exp(-t^2)$ for $\vert t\vert\geq C_n\sqrt{\log n}$ for some value $C_n=\Theta(1)$ such that this forms a probability distribution. One can show that $\Var(X)=\Theta(\log(n)/n)$ while the subgaussian norm is $K=\Theta(1)$ by \Cref{prop:subgaussian_equivalent}, so both these inequalities become tight in this case (for the probability lower bound, one may infinitesimally shift the probability mass at zero to the left without affecting the computations).

\subsection{Concentration Bounds}

\begin{thm}[Uniform Convergence, Theorem 4.10 and Example 5.24 of \citet*{Wainwright_2019}]
Let $\mathcal{F}$ be a class of Boolean functions from $\mathbb{R}^n\to \{0,1\}$ with VC dimension at most $d$. Then there exists absolute constants $C,c>0$ such that for any distribution $\mathcal{D}$ on $\mathbb{R}^n$, if $n\geq C d/\varepsilon^2$,
\begin{equation*}
    \Pr_{x_1,\ldots,x_n\sim \mathcal{D}}\left(\sup_{f\in \mathcal{F}}\left\vert\frac{\sum_{i=1}^n f(x_i)}{n}-\mathbb{E}[f(x)]\right\vert\geq \varepsilon\right) \leq 2\exp\left(-cn\varepsilon^2\right).
\end{equation*}
\end{thm}

\begin{fact}[see e.g. \citet*{Blumer}]
\label{fact:vc_dim_hs}
    The VC dimension of the set of intersections of two half-spaces in $\mathbb{R}^d$ is $O(d)$.
\end{fact}

An immediate consequence of the preceding two results is the following:
\begin{cor}
\label{cor:enough_thresh}
    Let $\mathcal{F}$ be a family of indicator functions for the intersection of two halfspaces in $\mathbb{R}^n$ such that every $f\in \mathcal{F}$ satisfies $\Pr(f(x)=1)\geq \alpha$. Then so long as $m\geq C(d+\log(1/\delta))/\alpha^2$, it holds that
    \begin{equation*}
        \Pr_{x_1,\ldots,x_m\sim \mathcal{D}}\left(\exists f\in \mathcal{F}: \text{at most $\frac{m\alpha}{2}$ of the $x_i$ satisfy $f(x_i)=1$}\right)\leq \delta.
    \end{equation*}
\end{cor}

\begin{thm}[General Bernstein Inequality, Corollary 2.8.3 of \citet*{vershynin}]
\label{thm:bernstein}
    There exists a universal constant $c>0$ such that the following holds. Suppose that $X_1,\ldots,X_m$ are i.i.d. random variables that are $K$-subexponential. Then for any $t<K$, it holds that
    \begin{equation*}
        \Pr\left(\left\vert \frac{1}{m}\sum_{i=1}^m X_i-\mathbb{E}[X]\right\vert\geq t\right)\leq 2\exp(-cmt^2/K^2).
    \end{equation*}

    In particular, there exists a constant $C>0$ such that for any $\delta>0$ and for any $t<K$, given $m\geq CK^2\log(2/\delta)/t^2$ i.i.d. samples $X_1,\ldots,X_m$ from a $K$-subexponential distribution, it holds with probability at least $1-\delta$ that
    \begin{equation*}
        \left\vert \frac{1}{m}\sum_{i=1}^m X_i-\mathbb{E}[X]\right\vert\leq t.
    \end{equation*}
\end{thm}

\begin{lem}[Gaussian Tail Bounds, see e.g. Proposition 2.1.2 of \citet*{vershynin}]
\label{lem:gaussian_exact}
    Let $X\sim \mathcal{N}(0,\sigma^2)$. Then for all $t> 0$,
    \begin{equation*}
        \frac{1}{\sqrt{2\pi}}\left(\frac{\sigma}{t}-\frac{\sigma^3}{t^3}\right)\exp(-t^2/2\sigma^2)\leq\Pr(X\geq t)\leq \frac{\sigma}{t\sqrt{2\pi}}\exp(-t^2/2\sigma^2).
    \end{equation*}
    In particular, if $t^2\geq 2\sigma^2$, then
    \begin{equation*}
        \frac{\sigma}{2t\sqrt{2\pi}}\exp(-t^2/2\sigma^2)\leq\Pr(X\geq t)\leq \frac{\sigma}{t\sqrt{2\pi}}\exp(-t^2/2\sigma^2).
    \end{equation*}
\end{lem}

\begin{defn}
    A real-valued random variable $X$ is \textbf{stochastically dominated} by a real-valued random variable $Y$ if there exists a coupling $(\tilde{X},\tilde{Y})$ such that $\tilde{X}$ has the same marginal law of $X$, $\tilde{Y}$ has the same marginal law of $Y$, and $\tilde{X}\leq \tilde{Y}$ holds almost surely. 
\end{defn}

\section{The Geometry and Moments of Regressors}
\label{sec:geometry_appendix}

In this section, we show how to use the moments of candidate regressors $\bm{v}$ to distinguish between vectors that are close to a true $\w_i$ and those that are far from all $\w_j$. The first key insight behind our approach, compared to the work of \citet{DBLP:conf/stoc/CherapanamjeriD23}, is that it suffices to reason about such moments conditional on events that have \emph{inverse polynomial probability} under \Cref{assumption:uncovered} and \Cref{assumption:bounded}, where the degree depends on $B,\Delta$. Indeed, while their algorithm required computing very high degree conditional moments given that the data has exceptionally small projection onto a subspace of dimension $k-1$, we will instead only reason about conditional moments given that a Gaussian $\bm{x}\in \mathbb{R}^n$ has only slightly large projection onto a given direction. Our main result of this section (\Cref{thm:separation}) will show that by using only a constant number of low-degree conditional moments (namely first and second moments) of the observations, we can effectively distinguish between close and far vectors regardless of the precise geometric configuration of $\w_1,\ldots,\w_k$ in relation to a candidate vector $\bm{v}$.

Throughout this section, we operate under \Cref{assumption:uncovered} and \Cref{assumption:bounded} even if not explicitly stated.

\subsection{The Geometry of Near Vectors}
Our first simple, but crucial, observation is that under \Cref{assumption:uncovered} and \Cref{assumption:bounded}, each vector $\w_j$ is the true maximizer with large probability, in a rather robust sense. In the work of \citet{DBLP:conf/stoc/CherapanamjeriD23}, this was shown with inverse exponential probability, but we show that a simple analysis actually yields inverse polynomial bounds for any fixed $B,\Delta>0$.

\begin{prop}
\label{prop:close_maximal}
    Suppose that \Cref{assumption:uncovered} and \Cref{assumption:bounded} hold. Let $\bm{0}\neq \bm{v}\in \mathbb{R}^n$. Then there exists a constant $C=C(B,\Delta)=\Theta(B^2/\Delta^2)$ such that the following holds. For any $t\geq C$, and any $\delta\in (0,1)$, define the event 
    \begin{equation*}
        A_{t,\overline{\bm{v}},\delta}=\left\{t\sqrt{\log(k/\delta)}\leq \bm{x}^T\bm{\overline{v}}\leq 2t\sqrt{\log(k/\delta)}\right\}.
    \end{equation*}
    Then
    \begin{enumerate}
    \item It holds that
                \begin{equation*}
\Pr\left(A_{t,\overline{\bm{v}},\delta}\right)= (\delta/k)^{\Theta(t^2)}.
                \end{equation*}
        \item Moreover, suppose $\bm{v}\in \mathbb{R}^n$ satisfies
    \begin{equation}
    \label{eq:close_vec}
        \|\bm{v}-\w_i\|\leq \Delta^3/8B^2
    \end{equation}
    for some $i\in [k]$. Then \begin{equation*}
            \Pr\left(i\neq\arg\max_{j\in [k]} \bm{x}^T\w_j+\eta_j\bigg\vert A_{t,\overline{\bm{v}},\delta}\right)\leq \delta.
        \end{equation*}
        \end{enumerate}
\end{prop}
\begin{proof}
    We first derive several geometric consequences of \Cref{eq:close_vec}. We claim that the index $i\in [k]$ attaining \Cref{eq:close_vec} is unique. To see this, observe that if there exists $j\neq i$ also satisfying \Cref{eq:close_vec}, then by basic properties of norms and projections, we would have
    \begin{align*}
        \|\w_i\|^2-\langle \w_i,\w_j\rangle&=\langle \w_i,\w_i-\w_j\rangle\\
        &\leq \|\w_i\|\|\w_i-\w_j\|\\
        &\leq B(\|\w_i-\bm{v}\|+\|\bm{v}-\w_j\|)\\
        &\leq 2B\Delta^3/8B^2\\
        &=\Delta^2/4,
    \end{align*}
    which violates \Cref{assumption:uncovered} upon rearrangement. Here, we use the fact $\Delta/B\leq 1$.

    Next, let $a_j \triangleq \langle \w_j,\overline{\bm{v}}\rangle$ for each $j\in [k]$. We claim that for $j\neq i$, it holds that
    \begin{equation*}
        \vert a_j\vert \leq \|\w_i\|-3\Delta^2/4B,
    \end{equation*}
    while $a_i\geq \|\w_i\|-\Delta^2/4B$. To see the former claim, 
    \begin{align*}
        a_j &= \langle \bm{w}_j,\overline{\bm{v}}\rangle\\
        &=\langle \bm{w}_j,\overline{\bm{w}}_i\rangle+\langle \bm{w}_j,\overline{\bm{v}}-\overline{\bm{w}}_i\rangle\\
        &\leq \|\w_i\|-\Delta^2/\|\w_i\|+\|\w_j\|\|\overline{\bm{v}}-\overline{\bm{w}}_i\|\\
        &\leq \|\w_i\|-\Delta^2/B+(2B/\Delta)\|\bm{v}-\w_i\|\\
        &\leq \|\w_i\| -3\Delta^2/4B,
    \end{align*}
    where we write $\overline{\w}_i=\w_i/\|\w_i\|$ and use \Cref{lem:contractions} with \Cref{eq:close_vec}, noting that $\|\w\|\geq \Delta$ and $\|\bm{v}\|\geq \|\w\|-\Delta^3/16B^2\geq \Delta/2$ since $\Delta/B\leq 1$.

    By a similar argument,
    \begin{align*}
        a_i &= \langle \bm{w}_i,\overline{\bm{v}}\rangle\\
        &=\langle \bm{w}_i,\overline{\bm{w}}_i\rangle+\langle \bm{w}_i,\overline{\bm{v}}-\overline{\bm{w}}_i\rangle\\
        &\geq \|\w_i\|-\|\w_i\|\|\overline{\bm{v}}-\overline{\bm{w}}_i\|\\
        &\geq \|\w_i\|-(2B/\Delta)\|\bm{v}-\w_i\|\\
        &\geq \|\w_i\| -\Delta^2/4B.
    \end{align*}

    Hence, we have $a_i-a_j\geq \Delta^2/2B$ for all $j\neq i$. In particular, on the event $A_{t,\overline{\bm{v}},\delta}$, for all $j\neq i$,
    \begin{equation}
    \label{eq:diff_v}
    \bm{x}^TP_{\overline{\bm{v}}}(\w_i)-\bm{x}^TP_{\overline{\bm{v}}}(\w_j)=(a_i-a_j)\bm{x}^T\overline{\bm{v}}\geq \frac{\Delta^2 t\sqrt{\log(k/\delta)}}{2B}.
    \end{equation}

    We now rewrite the sample observation by decomposing along the projection onto $\bm{v}$: we have
    \begin{align*}
        z &= \max_{j\in [k]}\left( \bm{x}^T\w_j+\eta_{j}\right)\\
        &=\max_{j\in [k]}\left( \bm{x}^TP_{\overline{\bm{v}}}(\w_j)+\bm{x}^TP_{\overline{\bm{v}}^{\perp}}(\w_j)+\eta_{j}\right)\\
        &=\max_{j\in [k]}\left( \bm{x}^TP_{\overline{\bm{v}}}(\w_j)+y_j\right),
    \end{align*}
    where we write $y_{j} \triangleq \bm{x}^TP_{\overline{\bm{v}}^{\perp}}(\w_j)+\eta_{j}$. Since the $\bm{x}^TP_{\overline{\bm{v}}}(\w_j)$ and $\eta_j$ are independent of $\bm{x}^T\overline{\bm{v}}$, the subgaussian norm triangle inequality implies that $\|y_j\|_{\psi_2}=O(B)$ for all $j$ via \Cref{assumption:bounded}. \Cref{fact:gaussian_maximal_tail} implies there exists an absolute constant $C>0$ such that
    \begin{equation*}
        \Pr\left(\sup_{j\in [k]} \vert y_j\vert\geq CB\sqrt{\log k}+x\bigg\vert A_{t,\overline{\bm{v}},\delta}\right)\leq \exp\left(-x^2/CB^2\right)
    \end{equation*}
    Taking $x=CB\sqrt{\log(1/\delta)}$, we thus find that for some other constant $C'>0$,
    \begin{equation}
    \label{eq:small_noise}
        \Pr\left(\sup_{j\in [k]} \vert y_j\vert\geq C'B\sqrt{\log (k/\delta)}\bigg\vert A_{t,\overline{\bm{v}},\delta}\right)\leq \delta
    \end{equation}

    It follows via \Cref{eq:diff_v} that on the complementary event of \Cref{eq:small_noise}, for any $j\neq i$,
    \begin{gather*}
        \bm{x}^T\w_i+\eta_i-\bm{x}^T\w_j+\eta_j\geq \frac{\Delta^2 t\sqrt{\log(k/\delta)}}{2B}-2\sup_{j\in [k]} \vert y_j\vert\geq \frac{\Delta^2 t\sqrt{\log(k/\delta)}}{2B}-2C'B\sqrt{\log (k/\delta)}.
    \end{gather*}
    So long as $t\geq C''B^2/\Delta^2$ for some other constant $C''>0$, this difference is strictly positive and thus $\w_i$ is the unique maximizer with probability at least $1-\delta$ conditional on $A_{t,\overline{\bm{v}},\delta}$.

    Finally, the claimed probability bound is a consequence of \Cref{lem:gaussian_exact}: the lower bound is obtained by considering the probability that $\bm{x}^T\overline{\bm{v}}\sim \mathcal{N}(0,1)$ exceeds $t\sqrt{\log(k/\delta)}$. On the other hand, the upper bound of \Cref{lem:gaussian_exact} applied with $2t\sqrt{\log(k/\delta)}$ is at most half the lower bound for any $t$ at least some absolute constant, hence $A_{t,\overline{\bm{v}},\delta}$ holds with the stated probability.
\end{proof}

For our later estimation of moments, we also record a simple lemma about the subgaussian norm of debiased observations on these conditional events:

\begin{lem}
\label{lem:cond_sg}
    Suppose that $\bm{v}\neq \bm{0}$ satisfies $\|\bm{v}\|\leq B$. Then for any $t\geq 1$ and $\delta>0$, the random variable
    \begin{equation*}
        Y \triangleq (\max_{j\in [k]} \bm{x}^T\w_j+\eta_j) - \bm{v}^T\bm{x}
    \end{equation*}
    is $O(Bt\sqrt{\log(k/\delta)})$ subgaussian on the event $A_{t,\overline{\bm{v}},\delta}$.
\end{lem}
\begin{proof}
    The claim follows straightforwardly from the fact that the subgaussian norm satisfies the triangle inequality. Simply observe that for any $\alpha$, on the event $A_{t,\overline{\bm{v}},\alpha}$, we have
    \begin{equation*}
        (\max_{j\in [k]} \bm{x}^T\w_j+\eta_j) - \bm{v}^T\bm{x}=\max_{j\in [k]}\left\{(a_j - \|v\|)\bm{x}^T\overline{\bm{v}}+y_j\right\} = \max_{j\in [k]}\{y_j\}+ Z,
    \end{equation*}
    where we use the same notation as in \Cref{prop:close_maximal} and $Z$ is a random variable (not independent of the $y_j$) that is surely bounded by $O(Bt\sqrt{\log(k/\alpha)})$ since $\vert a_j-\|\bm{v}\|\vert\leq 2B$ and using the definition of $A_{t,\overline{\bm{v}},\alpha}$. By \Cref{cor:max_subgaussian}, $\max_{j\in [k]} y_j$ is $O(B\sqrt{\log(k/\alpha))}$-subgaussian while $Z$ is trivially $O(Bt\sqrt{\log(k/\alpha)})$-subgaussian. Thus, on $A_{t,\overline{\bm{v}},\alpha}$,
    \begin{equation}
    \label{eq:cond_sg_norm}
        \left\|\max_{j\in [k]} \{\bm{x}^T\w_j+\eta_j\} - \bm{v}^T\bm{x}\right\|_{\psi_2}=O(Bt\sqrt{\log(k/\alpha)})
    \end{equation}
    by the triangle inequality for the subgaussian norm since we assume $t\geq 1$.
\end{proof}

Next, we show that the probability estimates established in \Cref{prop:close_maximal} ensure that vectors that are sufficiently close to some $\w_i$ will have similar low-order moments for debiased observations as $\w_i$ on these conditional events. These low-order moments correspond to the low-order moments of the standard linear regression model ($k=1$) after debiasing by $\w_i^T\bm{v}$. 
\begin{cor}
\label{cor:conditional_moments}
    There exists an absolute constant $C=C(B,\Delta)\geq 1$ and $c,c'>0$ sufficiently small constants such that the following holds for any $\gamma>0$ and any $t\geq C$ under \Cref{assumption:uncovered} and \Cref{assumption:bounded}. Suppose $\bm{v}\neq \bm{0}$ satisfies $\|\bm{v}\|\leq B$ and there exists $i\in [k]$ such that
    \begin{equation}
    \label{eq:close_vec_2}
        \|\bm{v}-\w_i\|\leq \min\left\{\Delta^3/8B^2,\frac{c\min\{\gamma,\sqrt{\gamma}\}}{t\sqrt{\log(k/\delta)}}\right\},
    \end{equation}
    where  $\delta=\frac{c'\gamma^2}{B^4t^4\log^2(Btk/c'\gamma)}$.

    Then on the event $A_{t,\overline{\bm{v}},\delta}$, the random variable $
        Y \triangleq (\max_{j\in [k]} \bm{x}^T\w_j+\eta_j) - \bm{v}^T\bm{x}$
    satisfies the following moment bounds: 
    \begin{enumerate}
        \item $\vert \mathbb{E}[Y\vert A_{t,\overline{\bm{v}},\delta}] \vert\leq \gamma$.
        \item $ \vert\mathbb{E}[Y^2\vert A_{t,\overline{\bm{v}},\delta}] - \sigma_i^2\vert\leq \gamma$,
        \item 
        $ \vert\mathbb{E}[Y_+^2\vert A_{t,\overline{\bm{v}},\delta}] - \sigma_{i,+}^2\vert\leq \gamma$.
    \end{enumerate}
\end{cor}
\begin{proof}
For the proof, all expectations and probabilities will be taken conditional on $A_{t,\overline{\bm{v}},\alpha}$, where $\alpha>0$ will later be set to $\delta$. We write
    \begin{align*}
        \max_{j\in [k]} \{\bm{x}^T\w_j+\eta_j\} &- \bm{v}^T\bm{x}=(\bm{x}^T(\w_i-\bm{v})+\eta_i)\\
        &+\underbrace{(\max_{j\neq i}\{\bm{x}^T(\w_j-\bm{v})+\eta_j\}-(\bm{x}^T(\w_i-\bm{v})+\eta_i))}_{E}\cdot \bm{1}(i\neq \arg\max_{j\in [k]} \bm{x}^T\w_j+\eta_j)
    \end{align*}
     so that
    \begin{equation}
        \vert \mathbb{E}[Y] -\mathbb{E}[\bm{x}^T(\w_i-\overline{\bm{v}})+\eta_i]\vert=\mathbb{E}[E\cdot \bm{1}(i\neq \arg\max_{j\in [k]} \bm{x}^T\w_j+\eta_j)]
    \end{equation}
    A nearly identical argument to \Cref{lem:cond_sg} (combined with the triangle inequality for the subgaussian norm) implies that $\|E\|_{\psi_2}=O(Bt\sqrt{\log(k/\alpha))}$, so 
    \begin{align*}
        \vert \mathbb{E}[E\cdot \bm{1}(i\neq \arg\max_{j\in [k]} \bm{x}^T\w_j+\eta_j)]\vert&\leq \sqrt{\mathbb{E}[E^2]\Pr(i\neq \arg\max_{j\in [k]} \bm{x}^T\w_j+\eta_j\vert A_{t,\overline{\bm{v}},\alpha})}\\
        &\leq O(Bt\sqrt{\log(k/\alpha)})\cdot \sqrt{\alpha},
    \end{align*}
    where we use \Cref{prop:subgaussian_equivalent} to bound the second moment and \Cref{prop:close_maximal} to bound the probability term.
    On the lefthand side, note that
    \begin{align*}
        \vert\mathbb{E}[\bm{x}^T(\bm{w}_i-\bm{v})+\eta_i]\vert  &=\vert \mathbb{E}[\bm{x}^T(P_{\overline{\bm{v}}}(\w_i)-\bm{v})+\bm{x}^TP_{\overline{\bm{v}}^{\perp}}(\w_i)+\eta_i]\vert\\
        &=\vert \mathbb{E}[\bm{x}^T(P_{\overline{\bm{v}}}(\w_i)-\bm{v})]\vert\\
        &\leq 2\|P_{\overline{\bm{v}}}(\w_i)-\bm{v}\|t\sqrt{\log(k/\alpha)}\\
        &\leq 2\|\w_i-\bm{v}\|t\sqrt{\log(k/\alpha)}
    \end{align*}
   Here, we use the fact that $\bm{x}^TP_{\overline{\bm{v}}^{\perp}}(\w_i)+\eta_j$ is a centered random variable that is independent of $A_{t,\overline{\bm{v}},\alpha}$, as well as the upper bound in the definition of  $A_{t,\overline{\bm{v}},\alpha}$ to bound the final expectation. We thus have
   \begin{equation*}
    \vert \mathbb{E}[Y]\vert\leq O(\|\w_i-\bm{v}\|t\sqrt{\log(k/\alpha)})+O(Bt\sqrt{\log(k/\alpha)})\cdot \sqrt{\alpha}
    \end{equation*}
    When $c'>0$ is chosen as a sufficiently small constant in the definition of $\alpha=\delta$, the second term is at most $\gamma/2$, while if $c>0$ is chosen after to be sufficiently small, the first term is also bounded by $\gamma/2$. Hence, the entire expectation is bounded by $\gamma$.

    For the second moment, we again have the decomposition
\begin{align*}
        &\left(\max_{j\in [k]} \{\bm{x}^T(\w_j-\bm{v})+\eta_j\}  \right)^2=(\bm{x}^T(\w_i-\bm{v})+\eta_i)^2\\
        &+\underbrace{\left(\left(\max_{j\neq i}\{\bm{x}^T(\w_j-\bm{v})+\eta_j\}\right)^2-(\bm{x}^T(\w_i-\bm{v})+\eta_i)^2\right)}_{F}\cdot \bm{1}(i\neq \arg\max_{j\in [k]} \bm{x}^T\w_j+\eta_j)
    \end{align*}

    We again have by \Cref{prop:subgaussian_equivalent} with the triangle inequality and Young's inequality that
    \begin{equation*}
        \mathbb{E}[F^2] = O(B^4t^4\log^2(k/\alpha)),
    \end{equation*}
    so an identical Cauchy-Schwarz argument implies that
    \begin{align*}
        \vert \mathbb{E}[Y^2]- \mathbb{E}[(\bm{x}^T(\w_i-\bm{v})+\eta_i)^2]\vert &\leq \sqrt{\mathbb{E}[F^2] \Pr(i\neq \arg \max_{j\in [k]} \{\bm{x}^T(\w_j-\bm{v})+\eta_j\})}\\
        &=O(B^2t^2\log(k/\alpha))\cdot \sqrt{\alpha}.
    \end{align*}
    On the lefthand side, we again write 
    \begin{equation*}
        \bm{x}^T(\w_i-\bm{v})+\eta_i = \bm{x}^T(P_{\overline{\bm{v}}}(\w_i)-\bm{v})+\bm{x}^TP_{\overline{\bm{v}}^{\perp}}(\w_i)+\eta_i.
    \end{equation*}
    Conditional on $A_{t,\overline{\bm{v}},\delta}$, the term $\bm{x}^TP_{\overline{\bm{v}}^{\perp}}(\w_i)$ is an independent Gaussian with variance at most $\|\w_i-\bm{v}\|^2$ while $\eta_j$ is independent with variance $\sigma_i^2$. Thus
    \begin{align*}
        \vert\mathbb{E}[(\bm{x}^T(\w_i-\bm{v})+\eta_i)^2]-\sigma_i^2\vert&\leq \|\w_i-\bm{v}\|^2+4\|\w_i-\bm{v}\|^2\cdot t^2\log(k/\alpha)\\
        &=O(\|\bm{w}_i-\bm{v}\|^2\cdot t^2\log(k/\alpha)).
    \end{align*}
    Thus, 
    \begin{equation*}
        \vert \mathbb{E}[Y^2]-\sigma_i^2\vert\leq O(\|\bm{w}_i-\bm{v}\|^2\cdot t^2\log(k/\alpha))+O(B^2t^2\log(k/\alpha))\cdot \sqrt{\alpha}\leq \gamma,
    \end{equation*}
   when $\alpha$ is taken to be $\delta$ so long as $c'>0$ was taken sufficiently small first to make the second term bounded by $\gamma/2$, and then $c>0$ was chosen as a sufficiently small constant to ensure the first term is at most $\gamma/2$. 
   The argument is identical for the case of $Y_+^2$, just replacing each term in the usual decomposition with the truncated versions. 
\end{proof}

\subsection{The Geometry of Far Vectors}
In this section, we consider the possible behaviors of moments of the samples conditioned on a particularly large Gaussian direction that is \emph{not} close to any true regressor. We will relate the behavior of the moments to the geometric properties of the $\w_i$ in relation to the candidate direction, and show that if all are far, then certain moment behaviors arise that cannot occur for much closer vectors as established in the previous section.

We start with the simple claim that if the conditional expectation of the (supposed) debiasing on the event $A_{t,\overline{\bm{v}},\delta}$ is small, then no true regressor $\bm{w}_i$ can have large projection in the direction $\bm{v}$. This simply holds by a direct stochastic domination argument on this event, since the observation must be much larger than the candidate direction can account for.
\begin{lem}[No Large Projections]
\label{lem:no_large_proj}
    Let $\gamma>0$ and let $\bm{0}\neq \bm{v}\in \mathbb{R}^n$. Suppose that 
    \begin{equation*}
        \vert \mathbb{E}[z-\bm{v}^T\bm{x}\vert A_{t,\overline{\bm{v}},\delta}]\vert \leq \gamma.
    \end{equation*}
    For each $j\in [k]$, let $a_j\in \mathbb{R}$ be such that $\w_j=a_j \overline{\bm{v}}+P_{\overline{\bm{v}}^{\perp}}(\w_i)$. Then it must hold for all $i\in [k]$ that
    \begin{equation}
    \label{eq:positive}
        a_i-\|\bm{v}\|\leq \frac{\gamma}{t\sqrt{\log(k/\delta)}}.
    \end{equation}
\end{lem}
\begin{proof}
The claim is straightforward from the definition of $A_{t,\overline{\bm{v}},\delta}$. Indeed, suppose that there exists $i\in [k]$ violating \Cref{eq:positive}. Note that the random variable $z-\bm{v}^T\bm{x}$ stochastically dominates the random variable $(a_i-\|\bm{v}\|)\bm{x}^T\overline{\bm{v}}+P_{\overline{\bm{v}}^{\perp}}(\w_i)^T\bm{x}+\eta_i$ conditional on $A_{t,\overline{\bm{v}},\delta}$. But on $A_{t,\overline{\bm{v}},\delta}$, it surely holds
\begin{equation*}
    (a_i-\|\bm{v}\|)\bm{x}^T\overline{\bm{v}}> \frac{\gamma}{t\sqrt{\log(k/\delta)}}\cdot t\sqrt{\log(k/\delta)}= \gamma.
\end{equation*}
Hence, $z-\bm{v}^T\bm{x}$ (strictly) stochastically dominates the random variable $\gamma+P_{\overline{\bm{v}}^{\perp}}(\w_i)^T\bm{x}+\eta_i$. By assumption, the latter two terms are centered random variables independent of $A_{t,\overline{\bm{v}},\delta}$. Hence, by basic properties of stochastic dominance,
\begin{equation*}
    \mathbb{E}[z-\bm{v}^T\bm{x}\vert A_{t,\overline{\bm{v}},\delta}]> \mathbb{E}[\gamma+P_{\overline{\bm{v}}^{\perp}}(\w_i)^T\bm{x}+\eta_i\vert A_{t,\overline{\bm{v}},\delta}]\geq \gamma,
\end{equation*}
contradicting the assumptions of the lemma.
\end{proof}

By a similar argument, we next claim that for any given $\bm{v}$, if there are multiple $\w_i$ that nearly project to $\bm{v}$, then the same conditional (debiased) mean of the observations must be noticeably positive on the event $A_{t,\overline{\bm{v}},\delta}$. This will follow because the (debiased) observation will stochastically dominate the maximum of two nearly centered random variables that are not completely positively correlated by the geometric condition \Cref{assumption:uncovered}, and thus will have noticeable expectation.

\begin{lem}
\label{lem:no_two_close}
    Let $\bm{v}\in \mathbb{R}^n$ be such that $\Delta\leq \|\bm{v}\|\leq B$ and $\min_{i\in [k]} \|\bm{v}-\w_i\|\geq \varepsilon$ where $\varepsilon^2<8\Delta^2$. Fix some $t\geq 1$ and suppose $\gamma>0$ is such that
    \begin{equation}
    \label{eq:kappa_bound}
        \kappa\triangleq \frac{\gamma}{t\sqrt{\log(k/\delta)}}\leq \min\left\{\frac{\Delta^2}{12B},\frac{\varepsilon}{4}\right\}.
    \end{equation}
    Suppose that there exists distinct indices $i\neq j\in[k]$ such that
    \begin{equation}
    \label{eq:two_close}
        \max\left\{\vert a_i-\|\bm{v}\|\vert,\vert a_j-\|\bm{v}\|\vert\right\} \leq  \kappa,
    \end{equation}
    where $a_{\ell}\in \mathbb{R}$ is such that $P_{\overline{\bm{v}}}(\w_{\ell})=a_{\ell}\overline{\bm{v}}$. Then under \Cref{assumption:uncovered} and \Cref{assumption:bounded}, it must hold that
    \begin{equation*}
        \mathbb{E}[z-\bm{v}^T\bm{x}\vert A_{t,\overline{\bm{v}},\delta}]\geq -2\gamma+\varepsilon/16.
    \end{equation*}
\end{lem}

\begin{proof}
    Let $\bm{u}=\w_i-P_{\overline{\bm{v}}}(\w_i)$, which we will shortly see is nonzero. By rotational invariance of Gaussians, we may assume that 
    \begin{gather*}
        \w_i = (\|\bm{v}\|+\kappa_i,\|\bm{u}\|,0),\\
        \w_j = (\|\bm{v}\|+\kappa_j,\alpha,\beta),
    \end{gather*}
    where we write these vectors in the orthonormal basis spanned by $\overline{\bm{v}},\overline{\bm{u}},$ and $\overline{P_{\{\overline{\bm{v}},\overline{\bm{u}}\}^{\perp}}(\w_j)}$. Here, $\max\{\vert \kappa_i\vert ,\vert \kappa_j\vert\}\leq \kappa$ by assumption.

    We now claim the following:
    \begin{claim} \label{claim:some_large_gaussian}
        Under the conditions of \Cref{lem:no_two_close}, one of the following two cases must hold:
\begin{enumerate}
    \item $\w_j$ has a large component orthogonal to $\mathsf{span}(\overline{\bm{v}},\overline{\bm{u}})$ i.e. $\vert \beta\vert\geq \varepsilon/4$, or
    \item The projections onto $\bm{u}$ are large and one is negative, i.e. $\| \bm{u}\|,\vert \alpha\vert\geq \varepsilon/2$ and $\alpha<0$.
\end{enumerate}
    \end{claim}
    \begin{proof}
Since $\|\bm{v}-\bm{w}_{\ell}\|\geq \varepsilon$ for all $\ell\in [k]$ by assumption, the triangle inequality implies that
\begin{gather*}
    \|\bm{u}\|+\kappa\geq \varepsilon,\\
    \vert \alpha\vert +\vert\beta\vert +\kappa \geq \varepsilon.
\end{gather*}
Since $\kappa\leq \varepsilon/4$ by assumption, we find that $\|\bm{u}\|\geq \varepsilon/2$. 

If the first conclusion that $\vert \beta\vert\leq \varepsilon/4$ fails, then the second inequality further implies that $\vert \alpha\vert\geq \varepsilon/2$, so it suffices to show that $\alpha$ is negative. This is a consequence of \Cref{assumption:uncovered} and \Cref{assumption:bounded}. Indeed, we can directly obtain that
\begin{align*}
    \langle \w_i,\w_j\rangle&= (\|\bm{v}\|+\kappa_i)(\|\bm{v}\|+\kappa_j)+\alpha\|\bm{u}\|\\
    &\geq \|\bm{v}\|^2+\alpha\|\bm{u}\|-3\kappa B.
\end{align*}
If $\alpha>0$, then we further have
\begin{align*}
\|\bm{v}\|^2+\alpha\|\bm{u}\|-3\kappa B &\geq \|\bm{v}\|^2+\min\{\alpha,\|\bm{u}\|\}^2-3\kappa B\\
&\geq \min\{\|\w_i\|^2,\|\w_j\|^2\}-\beta^2-6\kappa B\\
&> \min\{\|\w_i\|^2,\|\w_j\|^2\}-\Delta^2/2-6\kappa B,
\end{align*}
where the last inequality is since $\vert \beta\vert\leq \varepsilon/4$ by assumption and the original assumption $\varepsilon^2<8\Delta^2$.
But by \Cref{eq:kappa_bound}, this is strictly lower bounded by $\min\{\|\w_i\|,\|\w_j\|\}^2-\Delta^2$, contradicting \Cref{assumption:uncovered}.
\end{proof}

Given \Cref{claim:some_large_gaussian}, we can derive a suitable lower bound by Jensen's inequality and stochastic domination. Indeed, the random variable $z-\bm{v}^T\bm{w}$ stochastically dominates the random variable
\begin{equation*}
    \max\left\{(a_i-\|\bm{v}\|)\bm{x}^T\overline{\bm{v}}+\|\bm{u}\|\bm{x}^T\overline{\bm{u}}+\eta_i,(a_j-\|\bm{v}\|)\bm{x}^T\overline{\bm{v}}+\alpha\bm{x}^T\overline{\bm{u}}+\beta \bm{x}^T\overline{P_{\{\overline{\bm{v}},\overline{\bm{u}}\}^{\perp}}(\w_j)}+\eta_j \right\},
\end{equation*}
which on the event $A_{t,\overline{\bm{v}},\delta}$ stochastically dominates
\begin{equation*}
    -2\gamma+\max\left\{\|\bm{u}\|\bm{x}^T\overline{\bm{u}}+\eta_i,\alpha\bm{x}^T\overline{\bm{u}}+\beta \bm{x}^T\overline{P_{\{\overline{\bm{v}},\overline{\bm{u}}\}^{\perp}}(\w_j)}+\eta_j \right\},
\end{equation*}
using \Cref{eq:two_close}. Therefore, 
\begin{equation*}
    \mathbb{E}[z-\bm{v}^T\bm{x}\vert A_{t,\overline{\bm{v}},\delta}]\geq -2\gamma +\mathbb{E}\left[\max\left\{\|\bm{u}\|\bm{x}^T\overline{\bm{u}}+\eta_i,\alpha\bm{x}^T\overline{\bm{u}}+\beta \bm{x}^T\overline{P_{\{\overline{\bm{v}},\overline{\bm{u}}\}^{\perp}}(\w_j)}+\eta_j \right\}\bigg\vert A_{t,\overline{\bm{v}},\delta}\right],
\end{equation*}
so it suffices to show a suitable lower bound for the expectation. 

Observe that $\bm{x}^T\overline{\bm{u}}$ and $\bm{x}^T\overline{P_{\{\overline{\bm{v}},\overline{\bm{u}}\}^{\perp}}(\w_j)}$ are standard Gaussians (so long as $\beta\neq 0$) that are jointly independent even given the event $A_{t,\overline{\bm{v}},\delta}$ by orthogonality of Gaussian directions, and further are independent of $(\eta_i,\eta_j)$ by assumption. If the first case of \Cref{claim:some_large_gaussian} holds so that $\vert \beta\vert\geq \varepsilon/4$, then we can lower bound the above expectation using Jensen's inequality by:

\begin{align*}
    \mathbb{E}\left[\max\left\{\mathbb{E}[\|\bm{u}\|\bm{x}^T\overline{\bm{u}}+\eta_i],\mathbb{E}[\alpha\bm{x}^T\overline{\bm{u}}]+\beta \bm{x}^T\overline{P_{\{\overline{\bm{v}},\overline{\bm{u}}\}^{\perp}}(\w_j)}+\mathbb{E}[\eta_j] \right\}\bigg\vert A_{t,\overline{\bm{v}},\delta}\right]=\vert \beta\vert \mathbb{E}[g_+],
\end{align*}
where $g$ is standard Gaussian. Here, we use the fact that all other random variables are centered and independent of $g$ conditional on $A_{t,\overline{\bm{v}},\delta}$. But an easy calculation shows that $\mathbb{E}[g_+]=1/\sqrt{2\pi}$, and hence, we have a lower bound of at least $(\varepsilon/4)/\sqrt{2\pi}\geq \varepsilon/16$.

In the other case of \Cref{claim:some_large_gaussian}, a similar Jensen argument by integrating out the other independent random variables also yields a lower bound of 
\begin{equation*}
    \mathbb{E}\left[\max\left\{\|\bm{u}\|g,\alpha g\right\}\right],
\end{equation*}
where $g$ is standard Gaussian. But since \Cref{claim:some_large_gaussian} implies that $\alpha$ is negative and both coefficients have norm at least $\varepsilon/2$, this is at least
\begin{equation*}
(\varepsilon/2)\mathbb{E}\left[\vert g\vert\right],
\end{equation*}
where $g$ is standard Gaussian. Again, an easy calculation shows that $\mathbb{E}\left[\vert g\vert\right]=\sqrt{2/\pi}$, and thus we obtain a lower bound of $(\varepsilon/2)(\sqrt{2/\pi})\geq \varepsilon/16$ again. This concludes the proof.
\end{proof}

Thus, if $\gamma$ is small enough (i.e. much smaller than $\varepsilon$), the conditional mean will become noticeably positive if there are at least two $\varepsilon$-far regressors $\w_i$ and $\w_j$ that nonetheless project very close to $\bm{v}$.

The conclusion from \Cref{lem:no_large_proj} and \Cref{lem:no_two_close} is that, if for a fixed $t\gg 1$ large enough, the conditional (debiased) mean is very close to zero on the event $A_{t,\overline{\bm{v}},\delta}$, it must hold that there is \emph{at most one} true regressor whose projection onto the one-dimensional subspace spanned by $\bm{v}$ is very close to $\bm{v}$. We now reject the case that there are \emph{zero} regressors that project close to $\bm{v}$: because none of them can be noticeably positive by \Cref{lem:no_large_proj}, we can assume all of the projections, once one subtracts the norm of $\bm{v}$, are all noticeably negative. In that case, we will see that if we instead test using the slightly higher threshold parameter $4t\gg 1$, the conditional expectation of the mean must \emph{noticeably decrease}; this is because fixing all other independent randomness, the contribution coming from the $\bm{v}$ direction must then \emph{pointwise decrease} proportionally to $\gamma$. However for vectors $\bm{v}$ that are much closer to a true $\w_i$, this order of magnitude change will not happen by \Cref{cor:conditional_moments} since the conditional mean will be very close to zero for both thresholds. We now carry out this argument:
\begin{lem}
\label{lem:not_all_neg}
    Let $\bm{0}\neq \bm{v}\in \mathbb{R}^n$ and let $\gamma>0$. Let $t> 0$ be some constant, and suppose it holds that for all indices $i\in[k]$,
    \begin{equation}
    \label{eq:all_neg}
        a_i-\|\bm{v}\|\leq   \frac{-\gamma}{t\sqrt{\log(k/\delta)}},
    \end{equation}
    where $a_{\ell}\in \mathbb{R}$ is such that $P_{\overline{\bm{v}}}(\w_{\ell})=a_{\ell}\overline{\bm{v}}$. Then it must hold that
    \begin{equation*}
        \mathbb{E}[z-\bm{v}^T\bm{x}\vert A_{t,\overline{\bm{v}},\delta}]-\mathbb{E}[z-\bm{v}^T\bm{x}\vert A_{4t,\overline{\bm{v}},\delta}]\geq 2\gamma.
    \end{equation*}

    In particular,
    \begin{equation*}
        \max\left\{\left\vert \mathbb{E}[z-\bm{v}^T\bm{x}\vert A_{t,\overline{\bm{v}},\delta}]\right\vert,\left\vert \mathbb{E}[z-\bm{v}^T\bm{x}\vert A_{4t,\overline{\bm{v}},\delta}]\right\vert \right\vert\geq \gamma.
    \end{equation*}
\end{lem}
\begin{proof}
    Fix all the randomness that is independent of $\overline{\bm{v}}^T\bm{x}$; that is, fix all orthogonal Gaussian directions and the noise random variables, which does not affect the occurrence of any event $A_{s,\overline{\bm{v}},\delta}$ by independence. Consider the function $h(s)$ given by
    \begin{equation}
    \label{eq:thresh_deriv}
        h(s) = \max_{i\in [k]} \left\{(a_i-\|\bm{v}\|)(st\sqrt{\log(k/\delta)})+P_{\overline{\bm{v}}^{\perp}}(\w_i)^T\bm{x}+\eta_i\right\},
    \end{equation}
    where again, all the other randomness is viewed as fixed. We now observe that $h(s)$ is continuous and monotone decreasing by \Cref{eq:all_neg}. Therefore, $h$ is absolutely continuous and the derivative exists almost everywhere satisfying
    \begin{equation*}
        h'(s)\leq t\sqrt{\log(k/\delta)}\max_{i\in [k]} (a_i-\|\bm{v}\|)\leq -\gamma.
    \end{equation*}

    Now, we observe that taking expectations over the law of $s=\bm{x}^T\overline{\bm{v}}/(t\sqrt{\log(k/\delta)})
    \in [1,2]$ on $A_{t,\overline{\bm{v}},\delta}$, we have the pointwise bound
    \begin{equation*}
        \mathbb{E}_{s}[z-\bm{v}^T\bm{x}\vert A_{t,\overline{\bm{v}},\delta}]\geq h(2),
    \end{equation*}
    by the definition (in particular the upper bound) of $A_{t,\overline{\bm{v}},\delta}$ while similarly, using the lower bound in the definition of $A_{4t,\overline{\bm{v}},\delta}$,
    \begin{equation*}
        \mathbb{E}_{s}[z-\bm{v}^T\bm{x}\vert A_{4t,\overline{\bm{v}},\delta}]\leq h(4).
    \end{equation*}
    Thus, for any fixing of the other randomness, \Cref{eq:thresh_deriv} implies
    \begin{equation*}
        \mathbb{E}_{s}[z-\bm{v}^T\bm{x}\vert A_{t,\overline{\bm{v}},\delta}]\geq h(2)=-\int_2^4 h'(s)\mathrm{d}s+h(4)\geq 2\gamma+ h(4)\geq 2\gamma+\mathbb{E}_{s}[z-\bm{v}^T\bm{x}\vert A_{4t,\overline{\bm{v}},\delta}].
    \end{equation*}
    The first claim of the lemma then follows by taking expectations over all the remaining independent randomness on both sides. The last claim is an immediate consequence since if $ x-y\geq 2\gamma$, at least one of $\vert x\vert$ or $\vert y\vert$ must be at least $\gamma$.
\end{proof}

Thus, \Cref{lem:no_large_proj}, \Cref{lem:no_two_close}, and \Cref{lem:not_all_neg} collectively imply that  if a vector $\bm{v}$ passes a simple nearly unbiasedness test for just \emph{two} thresholds at slightly different scales, then it necessarily holds that there exists a \emph{unique} $\w_i$ such that $P_{\overline{\bm{v}}}(\bm{w}_i)\approx \bm{v}$.

In this case, if $\bm{v}$ is $\varepsilon$-far from all the regressors, then it is also $\varepsilon$-far from $\w_i$. In particular, the Gaussian component of the $i$th regressor orthogonal to $\bm{v}$ will have variance $\Omega(\varepsilon^2)$, so we expect some second moment statistic must be larger than it should be if $\bm{v}$ were actually close to $\w_i$, because we have extra independent Gaussian noise for that regressor. For technical reasons, we will consider the positive part of the second moment since this is an increasing function that can be analyzed more easily using stochastic dominance arguments. The tricky aspect is that unlike when the noise is assumed isotropic Gaussian, we cannot compare this quantity to ``what it is supposed to be'' since we do not know a priori what it is supposed to be, nor is the law of the noise necessarily symmetric to relate this quantity to the variance. However, it will turn out the subgaussianity condition controls any irregularities from the latter difficulty. 

We will require a technical lemma about how adding a Gaussian noise will increase the one-sided second moments. Define $g(z) = (z_+)^2$ and note that $g$ is a convex, nondecreasing function.

\begin{lem}
\label{lem:sm_increase}
    There exists an absolute constant $c>0$ such that the following holds for any $\varepsilon<1/10$. Suppose that $X$ is a centered, $K$-subgaussian random variable where $K\geq 1$. Suppose further that $Y\sim \mathcal{N}(0,\varepsilon^2)$ is independent of $X$. For any random variable $Z$ satisfying $\vert Z\vert \leq \zeta:=c\varepsilon^3/K^2\log(K/c\varepsilon)$ surely, it holds that
    \begin{equation*}
         \mathbb{E}[g(X+Y+Z)]\geq \mathbb{E}[g(X)]+\frac{c\varepsilon^4}{K^2\log(K/c\varepsilon)}
    \end{equation*}
\end{lem}
\begin{proof}
    We will use the better of two elementary lower bounds for $g(x+y)$. First, since $g$ is increasing and convex, and $X$ is independent of $Y$ and centered, monotonicity and Jensen's inequality implies
    \begin{equation}
    \label{eq:conv_bound_1}
        \mathbb{E}[g(X+Y+Z)]\geq \mathbb{E}[g(X+Y-\zeta)]\geq \mathbb{E}[g(Y-\zeta)]=\mathbb{E}[(Y-\zeta)_+^2]
    \end{equation}
    However, it is easy to verify that if $c>0$ is sufficiently small, this quantity is at least $\varepsilon^2/4$.

    Next, note that for any $x,y\in \mathbb{R}$, convexity of $g$ implies that
    \begin{equation*}
        g(x+y)\geq g(x)+g'(x)y+y^2\cdot \mathbf{1}(x\geq 0)\cdot \mathbf{1}(y\geq 0)=x_+^2+2x_+y+y_+^2\mathbf{1}(x\geq 0),
    \end{equation*}
    where the quadratic correction holds because the function restricts to the quadratic $z\mapsto z^2$ on $\mathbb{R}_+$. Taking expectations over $X,Y$, we also obtain
    \begin{align}
        \mathbb{E}[g(X+Y-\zeta)]&\geq \mathbb{E}[g(X)]-2\zeta\mathbb{E}[X_+]+\Pr(X\geq 0)\mathbb{E}[(Y-\zeta)_+^2]\\
        \label{eq:conv_bound_2}
        &\geq \mathbb{E}[g(X)]-2 \zeta\sqrt{\Var(X)}+\frac{c\Var(X)}{K^2\log(K^2/c\Var(X))}\cdot \varepsilon^2/4
    \end{align}
    where we use the independence of $X$ and $Y$ in the first inequality, and Cauchy-Schwarz in tandem with the computation in \Cref{eq:conv_bound_1} and \Cref{lem:sg_to_lbs} for the last inequality.

    We now have two cases. First, suppose that $\Var(X)\leq \varepsilon^2/100$. Then by \Cref{eq:conv_bound_1}, we are already done since $\mathbb{E}[g(X)]\leq \Var(X)\leq \varepsilon^2/100$ and so the increase is at least $\varepsilon^2/10\gg \varepsilon^4$ (since $\varepsilon< 1/10$). If $\Var(X)\geq \varepsilon^2/100$, then the lower bound of \Cref{eq:conv_bound_2} becomes
    \begin{equation*}
        \mathbb{E}[g(X+Y-\zeta)]\geq \mathbb{E}[g(X)]-\frac{\zeta\varepsilon}{5}+\frac{c\varepsilon^4}{K^2\log(K/c\varepsilon)},
    \end{equation*}
    for some other constant $c>0$. But so long as $\zeta\leq c'\varepsilon^3/K^2\log(K/c'\varepsilon)$ for some other absolute constant $c'>0$, this quantity is at least
    \begin{equation*}
        \frac{c''\varepsilon^4}{K^2\log(K/c''\varepsilon)},
    \end{equation*}
    as claimed. Thus, in either case we get the stated lower bound.
\end{proof}

Applying the previous result to our setting yields the following conclusion that the truncated second moment must noticeably increase for $\varepsilon$-far vectors $\bm{v}$ with the property that a true regressor $\w_i$ projects nearly onto $\bm{v}$.
\begin{cor}
\label{cor:inflated_sm}
    There exists an absolute constant $c>0$ such that the following holds for any  $\varepsilon<1/10, t\geq 1,\delta\in (0,1)$ with $B\geq 1$. Suppose that $\bm{v}$ is such that $\min_{j\in [k]} \|\bm{v}-\w_j\|\geq \varepsilon$, but there exists $i\in [k]$ such that
    \begin{equation*}
        \| \bm{v}-P_{\overline{\bm{v}}}(\w_i)\|\leq \frac{c\varepsilon^3}{2B^2\log(B/c\varepsilon)t\sqrt{\log(k/\delta)}}.
    \end{equation*}
    Then on the event $A_{t,\overline{\bm{v}},\delta}$, it holds that
    \begin{equation*}
     \mathbb{E}[(z-\bm{v}^T\bm{x})_+^2]\geq \sigma_{i,+}^2+\frac{c\varepsilon^4}{B^2\log(B/c\varepsilon)}
    \end{equation*}
\end{cor}
\begin{proof}
    As usual, the random variable $z-\bm{v}^T\bm{x}$ stochastically dominates the random variable $\bm{x}^T(P_{\overline{\bm{v}}}(\w_i)-\bm{v})+\bm{x}^TP_{\overline{\bm{v}}^{\perp}}(\w_i)+\eta_i$. In turn, this random variable stochastically dominates the random variable
    \begin{equation*}
        \frac{-c\varepsilon^3}{B^2\log(B/c\varepsilon)}+\bm{x}^TP_{\overline{\bm{v}}^{\perp}}(\w_i)+\eta_i
    \end{equation*}
    on the event $A_{t,\overline{\bm{v}},\delta}$. We may now apply \Cref{lem:sm_increase} with $X\triangleq\eta_i$, $Y\triangleq\bm{x}^TP_{\overline{\bm{v}}^{\perp}}(\w_i)$, and $Z\triangleq \frac{-c\varepsilon^3}{B^2\log(B/c\varepsilon)}$. We note that the variance of $Y$ is at least $\varepsilon^2/2$ by the farness assumption and the triangle inequality, and hence, \Cref{lem:sm_increase} implies via stochastic dominance and the fact $g$ is increasing that
    \begin{equation*}
        \mathbb{E}[(z-\bm{v}^T\bm{x})_+^2]\geq \mathbb{E}[X_+^2] +\frac{c'\varepsilon^4}{B^2\log(B/c'\varepsilon)}=\sigma_{i,+}^2+\frac{c'\varepsilon^4}{B^2\log(B/c'\varepsilon)},
    \end{equation*}
    where we slightly adjust the constant $c'>0$.
\end{proof} 

\subsection{Distinguishing Close and Far Vectors}

Using the results of the previous section, we can establish our main probabilistic result showing how to use low-degree moments of candidate vectors to distinguish $\varepsilon$-far vectors from sufficient close vectors to a true regressor. We will need the following easy geometric fact:

\begin{lem}
\label{lem:also_close}
    Let $\bm{v}\neq 0$ be such that 
    \begin{equation*}
        \|\bm{v}-P_{\overline{\bm{v}}}(\w_i)\|\leq \kappa_1.
    \end{equation*}
    Then if $\|\bm{u}-\w_i\|\leq \kappa_2$, it holds that
    \begin{equation*}
        \|\bm{v}-P_{\overline{\bm{v}}}(\bm{u})\|\leq \kappa_1+\kappa_2.
    \end{equation*}
\end{lem}
\begin{proof}
    Simply observe that
    \begin{align*}
        \|\bm{v}-P_{\overline{\bm{v}}}(\bm{u})\|&\leq \|\bm{v}-P_{\overline{\bm{v}}}(\w_i)\|+\|P_{\overline{\bm{v}}}(\w_i)-P_{\overline{\bm{v}}}(\bm{u})\|\\
        &\leq \|\bm{v}-P_{\overline{\bm{v}}}(\w_i)\|+\|\w_i-\bm{u}\|\\
        &\leq \kappa_1+\kappa_2,
    \end{align*}
    by the triangle inequality and the fact projections decrease distances.
\end{proof}

\begin{thm}
\label{thm:separation}
There exists a constant $C>0$ and small enough constants $c_1,c_2,c_3>0$ such that the following holds for any $\varepsilon<\Delta^2/10B^2$. Let $\bm{0}\neq \bm{v}\in \mathbb{R}^n$ be such that 
\begin{equation*}
    \min_{j\in [k]}\|\bm{v}-\w_j\|\geq \varepsilon.
\end{equation*}
Define the following parameters:
\begin{gather*}
    t = CB^2/\Delta^2,\\
    \gamma = \frac{c_1\varepsilon^4}{B^2\log(B/c_1\varepsilon)},\\
    \delta = \frac{c_2\gamma^2}{B^4t^4\log^2(Btk/c_2\gamma)}.
\end{gather*}
Then, at least one of the following two cases holds:
\begin{enumerate}
\item It holds that:
    \begin{equation}
    \label{eq:large_mean}
\max\left\{\left\vert \mathbb{E}\left[z-\bm{v}^T\bm{x}\bigg\vert A_{t,\overline{\bm{v}},\delta}\right]\right\vert,\left\vert \mathbb{E}\left[z-\bm{v}^T\bm{x}\bigg\vert A_{4t,\overline{\bm{v}},\delta}\right]\right\vert\right\}\geq \gamma,
\end{equation}
\item There exists a unique $i\in [k]$ such that for any $\bm{u}\in \mathbb{R}^n$ such that 
\begin{equation}
\label{eq:u_close}
    \|\bm{u}-\w_i\|\leq \frac{c_3\gamma}{t\sqrt{\log(k/\delta)}},
\end{equation}
\noindent the following inequalities hold:
\begin{gather}
\label{eq:close_proj_v}
    \|\bm{v}-P_{\overline{\bm{v}}}(\bm{u})\|\leq \frac{2\gamma}{t\sqrt{\log(k/\delta)}},\\
    \label{eq:low_mean}
    \max\left\{\left\vert \mathbb{E}\left[z-\bm{u}^T\bm{x}\bigg\vert A_{t,\overline{\bm{u}},\delta}\right]\right\vert,\left\vert \mathbb{E}\left[z-\bm{v}^T\bm{x}\bigg\vert A_{4t,\overline{\bm{u}},\delta}\right]\right\vert\right\}\leq \frac{\gamma}{2},\\
    \label{eq:smaller_sm}
    \mathbb{E}[(z-\bm{u}^T\bm{x})_+^2\vert A_{t,\overline{\bm{u}},\delta}]\leq \mathbb{E}[(z-\bm{v}^T\bm{x})_+^2\vert A_{t,\overline{\bm{v}},\delta}]-\frac{\gamma}{2}.
\end{gather}
\end{enumerate}
\end{thm}
\begin{proof}
    This is a straightfoward consequence of all the preceding results of this section. If \Cref{eq:large_mean} does not hold, it follows by \Cref{lem:no_large_proj}, \Cref{lem:no_two_close}, and \Cref{lem:not_all_neg} that there exists a unique $i\in [k]$ such that
    \begin{equation*}
        \|P_{\overline{\bm{v}}}(\w_i)-\bm{v}\|\leq \frac{\gamma}{t\sqrt{\log(k/\delta)}},
    \end{equation*}

    For any $\bm{u}$ satisfying \Cref{eq:u_close} with $c_3\leq 1$, \Cref{eq:close_proj_v} follows from \Cref{lem:also_close}. Moreover, \Cref{eq:low_mean} follows as a consequence of \Cref{cor:conditional_moments} when $c_2,c_3>0$ are small enough constants. Finally, \Cref{eq:smaller_sm} follows from \Cref{cor:conditional_moments} and \Cref{cor:inflated_sm} when $c_1,c_2,c_3>0$ are small enough constants.
\end{proof}

\section{Finding an Approximate Subspace}
\label{sec:subspace_appendix}
In this section, we provide for completeness the argument of \citet{DBLP:conf/stoc/CherapanamjeriD23} for how to recover a subspace $V$ of dimension $k$ such that every $\w_i$ approximately lies in $V$. For the convenience of the reader, we provide the relevant details of their sketched argument given their structural results. Their key observation is the following. Define the following empirical, weighted second-moment matrix:
\begin{equation}
    M \triangleq \mathbb{E}[\max\{0,z\}^2\bm{x}\bm{x}^T],
\end{equation}
where $(\bm{x},z)$ is a sample from the linear regression with self-selection bias model of \Cref{eq:lin_reg_ss}, i.e.
\begin{equation*}
    z = \max_{i\in [k]} \w_i^T\bm{x}+\eta_i.
\end{equation*}

They establish the following crucial facts about $M$. The first is obtained by using Gaussian integration by parts to establish a closed form lower bound on the quadratic form induced by $M$ on important subspaces.
\begin{lem}[Lemma 7 of \citet{DBLP:conf/stoc/CherapanamjeriD23}]
\label{lem:M_lb}
    Define 
    \begin{equation*}
        p_i\triangleq \Pr(i=\arg\max_{j\in [k]}\w_j^T\bm{x}+\eta_j \land z\geq 0),
    \end{equation*} 
    where the probability is taken over $\bm{x}\sim \mathcal{N}(0,I)$ and the noise $\bm{\eta}$. Then for all unit vectors $\bm{v}\in \mathsf{span}(\w_1,\ldots,\w_k)$, it holds that
    \begin{equation*}
        \bm{v}^TM\bm{v}\geq \mathbb{E}[\max\{0,z\}^2] + 2\sum_{i=1}^k p_i(\bm{v}^T\w_i)^2.
    \end{equation*}
    Meanwhile, for all unit vectors $\bm{v}\in \mathsf{span}(\w_1,\ldots,\w_k)^{\perp}$, it holds that
    \begin{equation*}
        \bm{v}^TM\bm{v}= \mathbb{E}[\max\{0,z\}^2].
    \end{equation*}
\end{lem}
We note that while they assume that the noise is isotropic Gaussian, their argument holds as stated just under the assumption that $\bm{\eta}$ is independent of $\bm{x}$ since they only apply Gaussian integration by parts in the $\bm{x}$ variables, so one may assume the noise is fixed and then take expectations over the noise afterwards. The distribution $\mathcal{D}$ of $\bm{\eta}$ only affects the $p_i$ terms and the expectation determining the identity component.

Next, they show that the span of the top $\ell\leq k$ eigenvectors must approximately contain each of the $\w_i$ in the following sense: 

\begin{lem}[Lemma 10 of \citet{DBLP:conf/stoc/CherapanamjeriD23}]
\label{lem:reg_in_top}
For any $\varepsilon\in (0,1/2)$, let $\{\bm{v}_i\}_{i=1}^{\ell}$ denote the orthonormal set of eigenvectors for $M$ such that the corresponding eigenvalues $\lambda_1\geq \ldots\geq \lambda_{\ell}$ satisfy\footnote{Note that $\ell\leq k$ by the Courant-Fischer minimax theorem (see e.g. Corollary III.1.2 of \citet*{bhatia1997matrix}) since there exists a subspace of dimension $\dim(\mathsf{span}(\w_1,\ldots,\w_k))$ on which the quadratic form strictly exceeds $\mathbb{E}[\max\{0,z\}^2]$ for all unit vectors, while there is a subspace with identical codimension where this is equality for all unit vectors.}
\begin{equation*}
\lambda_{i}\geq \mathbb{E}[\max\{0,z\}^2]+\min_{j\in [k]} \frac{p_j\varepsilon^2}{2}.
\end{equation*}

Then if $V=\mathsf{span}(\bm{v}_1,\ldots,\bm{v}_{\ell})$, it holds for all $i\in [k]$ that\footnote{Their result is stated in terms of the relative $\ell_2$ mass of $\w_i$, but this is bounded by $B$ by \Cref{assumption:bounded}.}
\begin{equation*}
    \|\w_i-P_V(\w_i)\|\leq B\varepsilon.
\end{equation*}
\end{lem}

Next, they show that an empirical estimator $\widehat{M}$ of $M$ enjoys good spectral approximation guarantees. Their proof relies on existing probabilistic results of \citet*{DBLP:journals/corr/YiCS16} whose statements assume Gaussianity, so we provide a similar, but elementary and self-contained argument.

\begin{lem}
\label{lem:trnc_close}
    Let $\varepsilon> 0$ and let $\mathcal{E}$ denote the event that $\max\{0,z\}\leq CB\sqrt{\log(nB\log(k)/\varepsilon)}$ for a sufficiently large constant $C>0$. Then it holds that
    \begin{equation*}
        \|M - \mathbb{E}[\max\{0,z\}^2\bm{x}\bm{x}^T\mathbf{1}(\mathcal{E})]\|_{\mathsf{op}}\leq \varepsilon/2.
    \end{equation*}
\end{lem}
\begin{proof}
    By definition, it suffices to bound 
    \begin{equation*}
        \|\mathbb{E}[\max\{0,z\}^2\bm{x}\bm{x}^T\mathbf{1}(\mathcal{E}^c)]\|_{\mathsf{op}}.
    \end{equation*}

    But this can be done easily using Jensen's inequality and repeatedly applying the Cauchy-Schwarz inequality as follows:
\begin{align*}
    \|\mathbb{E}[\max\{0,z\}^2\bm{x}\bm{x}^T\mathbf{1}(\mathcal{E}^c)]\|_{\mathsf{op}}&\leq \mathbb{E}[\max\{0,z\}^2\|\bm{x}\|^2\mathbf{1}(\mathcal{E}^c)]\\
    &\leq \mathbb{E}[\max\{0,z\}^4\|\bm{x}\|^4]^{1/2}\sqrt{\Pr(\mathcal{E}^c)}\\
    &\leq \mathbb{E}[\max\{0,y\}^8]^{1/4}\mathbb{E}[\|\bm{x}\|^8]^{1/4}\sqrt{\Pr(\mathcal{E}^c)}\\
    &\leq C'B^2n\log(k)\sqrt{\Pr(\mathcal{E}^c)},
\end{align*}
for a sufficiently large constant $C'>0$. In this calculation, we used the fact that $\|\bm{x}\|-\sqrt{n}$ is $O(1)$-subgaussian by \cite[Theorem 3.1.1]{vershynin}, and hence by \Cref{prop:subgaussian_equivalent}, we have
\begin{equation*}
    \mathbb{E}[\|\bm{x}\|^8]^{1/8}\leq \sqrt{n}+\mathbb{E}[(\|\bm{x}\|-\sqrt{n})^8]^{1/8}\leq \sqrt{n}+O(1).
\end{equation*}
Similarly, we also used \Cref{cor:max_subgaussian} and \Cref{prop:subgaussian_equivalent} to see that
\begin{equation*}
    \mathbb{E}[\max\{0,z\}^8]\leq O(B^8\log^4(k))
\end{equation*}
since $z$ is the maximum of $k$ random variables that are $O(B)$-subgaussian by \Cref{assumption:bounded}. By \Cref{fact:gaussian_maximal_tail}, we also have
\begin{equation*}
    \Pr(\mathcal{E}^c)\leq \exp\left(-C''\log(Bn\log(k)/\varepsilon)\right)\leq \frac{\varepsilon^2}{4C'^2B^4n^2\log^2(k)},
\end{equation*}
so long as the constant $C>0$ was chosen large enough. This completes the proof.
\end{proof}

From now on, define the matrix 
\begin{equation*}
    M'=\mathbb{E}[\max\{0,z\}^2\bm{x}\bm{x}^T\mathbf{1}(\mathcal{E})],
\end{equation*}
where $\mathcal{E}$ is the event from \Cref{lem:trnc_close}. By \Cref{lem:trnc_close}, one can obtain a $\varepsilon$ spectral approximation of $M$ via a $\varepsilon/2$ spectral approximation of $M'$. To establish this matrix concentration, we first require a lower bound on the quadratic form induced by the truncated random vector $\mathbf{1}(\mathcal{E})\cdot \max\{0,z\}\bm{x}$:

\begin{lem}
\label{lem:trnc_var_lb}
    There exists a constant $c>0$ such that
    \begin{equation*}
        \mathbb{E}[\max\{0,z\}^2]\geq c\frac{\Delta^4\log(k)}{B^2}.
    \end{equation*}

    Therefore, so long as $\varepsilon\leq c\Delta^4\log(k)/2B^2$, for any $\bm{v}\in \mathcal{S}^{n-1}$,
    \begin{equation*}
        \bm{v}^TM'\bm{v}\geq c\frac{\Delta^4\log(k)}{2B^2}.
    \end{equation*}
\end{lem}
\begin{proof}
    By Jensen's inequality,
    \begin{equation*}
        \mathbb{E}[\max\{0,z\}^2]\geq \mathbb{E}[\max\{0,z\}]^2\geq \mathbb{E}[\max\{0,\w_1^T\bm{x},\ldots, \w_k^T\bm{x}\}]^2,
    \end{equation*}
    where we use the fact the noise is centered and independent of $\bm{x}$, the max function is convex, and the expectation is always nonnegative. Then by Sudakov's inequality (see e.g. Theorem 6.5 of \citet*{vanhandel}), there exists an absolute constant $c'>0$ such that
    \begin{equation*}
        \mathbb{E}[\max\{0,\w_1^T\bm{x},\ldots, \w_k^T\bm{x}\}]\geq c'\frac{\Delta^2}{B}\sqrt{\log(k)}.
    \end{equation*}
    where we use the separation lower bound of \Cref{lem:sep} to deduce the covering number of the Gaussian family  is exactly $k$ at scale $\Delta^2/2B$.

    The last claim is an immediate consequence of \Cref{lem:M_lb} and \Cref{lem:trnc_close} since
    \begin{equation*}
        \bm{v}^TM'\bm{v}\geq \bm{v}^TM\bm{v}-\|M-M'\|_{\mathsf{op}}\geq c\frac{\Delta^4\log(k)}{B^2}-c\frac{\Delta^4\log(k)}{2B^2}\geq c\frac{\Delta^4\log(k)}{2B^2}.
    \end{equation*}
\end{proof}
 We now show how to obtain the required matrix concentration of the empirical, truncated weighted second moment estimator. We will require the following easy result bounding the subgaussian constant of these random vectors:
\begin{lem}
\label{lem:trnc_vec_sg}
    There exists an absolute constant $C>0$ such that for any $\bm{v}\in \mathcal{S}^{n-1}$,
    \begin{equation*}
        \|\langle \max\{0,z\}\bm{x} \mathbf{1}(\mathcal{E}),\bm{v}\rangle\|_{\psi_2}\leq CB\sqrt{\log(nB\log(k)/\varepsilon)}
    \end{equation*}
\end{lem}
\begin{proof}
    By definition, it suffices to show that there exists a constant $C>0$ such that 
    \begin{equation*}
        \mathbb{E}\left[\exp\left(\frac{\max\{0,z\}^2 (\bm{v}^T\bm{x})^2\mathbf{1}(\mathcal{E})}{CB^2\log(nB\log(k)/\varepsilon)}\right)\right]\leq 2.
    \end{equation*}
    But clearly on the event $\mathcal{E}$, it holds that for some constant $C'>0$,
\begin{equation*}
    \frac{\max\{0,z\}^2}{C'B^2\log(nB\log(k)/\varepsilon)}\leq 1,
\end{equation*}
and so
\begin{equation*}
    \mathbb{E}\left[\exp\left(\frac{\max\{0,z\}^2 (\bm{v}^T\bm{x})^2\mathbf{1}(\mathcal{E})}{CB^2\log(nB\log(k)/\varepsilon)}\right)\right]\leq \mathbb{E}[\exp((\bm{v}^T\bm{x})^2/C'')]\leq 2
\end{equation*}
so long as $C\geq 1$ is some large enough constant since $\bm{v}^T\bm{x}$ is standard Gaussian and thus has constant subgaussian norm.
\end{proof}

With these results in hand, our desired matrix concentration will follow from standard results:

\begin{lem}
\label{lem:mat_concentration}
    Suppose that $\varepsilon\leq c\Delta^4\log(k)/B^2$ for a small enough constant $c>0$. Then, given at least  
    \begin{equation*}
    m\triangleq \frac{CB^{12}\log^4(nB\log(k)/\varepsilon)\cdot (n+\sqrt{\log(2/\lambda)})}{\log^2(k)\Delta^8\varepsilon^2}
    \end{equation*}
    i.i.d. samples $(\bm{x}_{\ell},z_{\ell})$ from \Cref{eq:lin_reg_ss}, it holds with probability at least $1-\lambda$ that
    \begin{equation*}
        \left\|\frac{1}{m}\sum_{\ell=1}^m \max\{0,z_{\ell}\}^2\bm{x}_{\ell}\bm{x}_{\ell}^T\mathbf{1}(\mathcal{E})-M\right\|_{\mathsf{op}}\leq \varepsilon.
    \end{equation*}
\end{lem}
\begin{proof}
    As a consequence of \Cref{lem:trnc_var_lb} and \Cref{lem:trnc_vec_sg}, there is some constant $C>0$ such that 
    \begin{equation*}
        \sup_{\bm{v}\in \mathbb{R}^{n}} \frac{\|\langle \max\{0,z\}X \mathbf{1}(\mathcal{E}),\bm{v}\rangle\|_{\psi_2}}{\sqrt{\mathbb{E}[\mathbf{1}(\mathcal{E})(\max\{0,z\}\bm{v}^T\bm{x})^2]}}\leq \frac{CB^2\sqrt{\log(nB\log(k)/\varepsilon)}}{\Delta^2\sqrt{\log(k)}}:=K.
    \end{equation*}

    By \cite[Theorem 4.7.1]{vershynin} (see \cite[Exercise 4.7.3]{vershynin}), it follows that there exists another constant $C'$ such that with probability at least $1-\lambda$, it holds that
    \begin{equation*}
        \left\|\frac{1}{m}\sum_{\ell=1}^m \max\{0,z_{\ell}\}^2\bm{x}_{\ell}\bm{x}_{\ell}^T\mathbf{1}(\mathcal{E})-M'\right\|_{\mathsf{op}}\leq C'K^2\left(\sqrt{\frac{n+\sqrt{\log(2/\lambda)}}{m}}+\frac{n+\sqrt{\log(2/\delta)}}{m}\right)\|M'\|_{\mathsf{op}}.
    \end{equation*}
    But observe that 
    \begin{equation*}
        \|M'\|_{\mathsf{op}}\leq CB^2\log(nB\log(k)/\varepsilon)
    \end{equation*}
    for some constant $C>0$ by the definition of $\mathbf{1}(\mathcal{E})$. This is because if $a\leq b$, then for any $\bm{x}\in \mathbb{R}^n$,
    \begin{equation*}
        a\bm{x}\bm{x}^T\preceq b\bm{x}\bm{x}^T,
    \end{equation*}
    and therefore
    \begin{equation*}
        \bm{0}\preceq M'=\mathbb{E}[\max\{0,z\}^2\bm{x}\bm{x}^T\mathbf{1}(\mathcal{E})]\preceq CB^2\log(nB\log(k)/\varepsilon)\mathbb{E}[\bm{x}\bm{x}^T]=CB^2\log(nB\log(k)/\varepsilon)\cdot I_n
    \end{equation*}
    
    Thus, so long as 
    \begin{equation*}
        m\geq O\left(\frac{K^4\|M\|_{\mathsf{op}}^2}{\varepsilon^2}\right)\cdot (n+\sqrt{\log(2/\lambda)})=O\left(\frac{B^{12}\log^4(nB\log(k)/\varepsilon)\cdot (n+\sqrt{\log(2/\delta)})}{\log^2(k)\Delta^8\varepsilon^2}\right),
    \end{equation*}
    we will have 
    \begin{equation*}
        \left\|\frac{1}{m}\sum_{\ell=1}^m \max\{0,z_{\ell}\}^2\bm{x}_{\ell}\bm{x}_{\ell}^T\mathbf{1}(\mathcal{E})-M'\right\|_{\mathsf{op}}\leq \varepsilon/2.
    \end{equation*}
    Combining with \Cref{lem:trnc_close}, we deduce
    \begin{equation*}
        \left\|\frac{1}{m}\sum_{\ell=1}^m \max\{0,z_{\ell}\}^2\bm{x}_{\ell}\bm{x}_{\ell}^T\mathbf{1}(\mathcal{E})-M\right\|_{\mathsf{op}}\leq \varepsilon
    \end{equation*}
    with probability at least $1-\lambda$.
\end{proof}

Below, we write
\begin{equation*}
    \widehat{M}=\frac{1}{m}\sum_{\ell=1}^m \max\{0,z_{\ell}\}^2\bm{x}_{\ell}\bm{x}_{\ell}^T\mathbf{1}(\mathcal{E}),
\end{equation*}
where the $\varepsilon>0$ term is implicit in the definition of $\mathcal{E}$ for a desired spectral approximation. We can finally conclude with the following:
\begin{thm}
\label{thm:subspace}
    There exists constants $c,c'>0$ such that for any $\varepsilon\leq c\Delta^4\log(k)/B^2$, if
    \begin{equation}
    \label{eq:mat_approx}
        \left\|\widehat{M}-M\right\|_{\mathsf{op}}\leq c'\frac{\varepsilon^2}{k^{\Theta(B^4/\Delta^4)}B^3}:=\eta,
    \end{equation}
    then if $U$ is the span of the top $k$ eigenvectors of $\widehat{M}$ (chosen arbitrarily if with multiplicity), it holds for all $i\in [k]$ that
    \begin{equation*}
    \|\w_i-P_U(\w_i)\|\leq \varepsilon.
    \end{equation*}

    In particular, this holds with probability at least $1-\lambda$ so long as 
    \begin{equation*}
        m\geq \frac{CB^{34}k^{\Theta(B^4/\Delta^4)}\log^4(knB\log(k)/\varepsilon)\cdot (n+\sqrt{\log(2/\lambda)})}{\log^2(k)\Delta^{24}\varepsilon^4}
    \end{equation*}
\end{thm}
\begin{proof}
    By \Cref{lem:reg_in_top}, it holds that 
    \begin{equation*}
        \|\w_i-P_{V}(\w_i)\|\leq \varepsilon/2
    \end{equation*}
    where $V$ is the subspace spanned by $\ell\leq k$ eigenvectors of $M$ with eigenvalue at least
    \begin{equation*}
        \mathbb{E}[\max\{0,z\}^2]+\frac{\varepsilon^2\min_{j\in [k]}p_j}{8B^2}\geq \mathbb{E}[\max\{0,z\}^2]+\frac{\varepsilon}{8B^2k^{\Theta(B^4/\Delta^4)}},
    \end{equation*}
    where we use \Cref{prop:close_maximal} to provide the uniform lower bound on each $p_i$ under \Cref{assumption:uncovered} and \Cref{assumption:bounded}.

    Next, let $T$ denote the $n-\ell'$ dimensional subspace spanned by eigenvectors of $\widehat{M}$ with eigenvalue at most
    \begin{equation*}
        \mathbb{E}[\{0,z\}^2]+\eta.
    \end{equation*}
    By the Weyl inequalities (see e.g. Theorem III.2.1 of \citet*{bhatia1997matrix}), \Cref{eq:mat_approx}, and the fact that the eigenvalue multiplicity of $\mathbb{E}[\max\{0,z\}^2]$ of $M$ is at least $n-k$, we observe that $\ell'\leq k$. Moreover, by the Davis-Kahan $\sin\Theta$ theorem (see e.g. Theorem VII.3.1. of \citet*{bhatia1997matrix}), the Weyl inequalities, and \Cref{eq:mat_approx} again, it holds that
    \begin{equation*}
        \|P_TP_V\|_{\mathsf{op}}\leq \frac{\eta}{\frac{\varepsilon}{8B^2k^{\Theta(B^4/\Delta^4)}}-\eta}\leq \frac{\varepsilon}{2B},
    \end{equation*}
    so long as $c'$ was chosen small enough.
But by definition, this implies that every vector in $V$ with norm at most $B$ is within $\varepsilon/2$ of a vector in $T^{\perp}$, which is spanned by the at most $\ell'\leq k$ top eigenvectors of $\widehat{M}$. Since $T^{\perp}\subseteq U$, the claim follows. The final sample complexity bound then follows from \Cref{lem:mat_concentration}.
\end{proof}

\section{Complete Algorithm and Sample Complexity Analysis}
\label{sec:final_appendix}
Given the results of the previous sections, we may assume access to a $k$-dimensional subspace $V$ that approximately contains each of the $\w_1,\ldots,\w_k$. After we obtain a collection of vectors in a fine enough net of $V$ that pass a simple first moment test on suitable conditional events, which includes all very close vectors to a $\bm{w}_i$, we can extract close vectors using the following procedure. We will iteratively select a vector with minimum empirical truncated second moment. By \Cref{thm:separation}, we will be able to deduce that such a vector must be close to a true $\w_i$ since any very close vector to $\w_i$ will have noticeably smaller truncated second moment than a far vector that only possibly passed the first moment test because $\w_i$ projected exceptionally close to it. 

This procedure can thus be shown to find a close vector $\bm{v}^*$ to some $\w_i$, but how do we make further progress towards finding the others? Recall that the noise components for each $\w_i$ can be different and are unknown, so the exact value of the minimizers do not appear algorithmically useful. We will instead use the geometry of the true regressors as established in the previous section to eliminate all other vectors that ``correspond'' to $\w_i$ while keeping all vectors that are $\varepsilon$-close to $\w_j$ for any $j$ that has not been selected yet. To do this, we use the observation that any vector $\bm{v}$ corresponding to $\w_i$ in the net that has not been eliminated so far either is close to $\bm{v}^*$ by the triangle inequality, or is such that $\w_i$ projects quite close to it. While we do not know $\w_i$ exactly, if the net is sufficiently fine, there will be vector in the net that is close to $\bm{v}^*$ that is very close to $\w_i$. The characterization in the second part of \Cref{thm:separation} will imply there is a close vector to $\bm{v}^*$ with very accurate projection onto $\bm{v}$, so we can simply eliminate any $\bm{v}$ such that there exists a close vector to $\bm{v}^*$ that projects nearly onto $\bm{v}$. We will also show that this deletion step cannot affect vectors that are $\varepsilon$-close to $\w_j$ for $j\neq i$ by \Cref{assumption:uncovered}. The upshot is that in each iteration, we find a close vector to a $\w_i$, delete all vectors that ``correspond'' to $\w_i$, and keep all vectors that are close to a different $\w_j$ for a $j\in [k]$ that has not already been found.

We now present the full algorithm as \Cref{alg:alg_one}. Throughout the rest of this section, given $0< \varepsilon\leq \Delta^3/10B^3$, define $t,\gamma,\delta>0$ as in \Cref{thm:separation}:

\begin{gather}
    t = CB^2/\Delta^2\\
    \gamma = \frac{c_1\varepsilon^4}{B^2\log(B/c_1\varepsilon)}\\
    \delta = \frac{c_2\gamma^2}{B^4t^4\log^2(Btk/c_2\gamma)}.
\end{gather}

Given a subspace $V\subseteq \mathbb{R}^n$ and scalars $0\leq r<R$, let $B_V(r,R)$ denote the annulus $\{\bm{w}\in V: r\leq \|\bm{v}\|\leq R\}$ in $V$. Given $m$ i.i.d. samples $\mathcal{T}=(\bm{x}_{\ell},z_{\ell})_{\ell=1}^m$ from the model, a vector $\bm{v}\neq \bm{0}$, and parameters $s,\alpha$, define the following sets and statistics:
\begin{gather*}
    \mathcal{A}_{s,\overline{\bm{v}},\alpha} = \left\{(\bm{x}_{\ell},z_{\ell})\in \mathcal{T}: s\sqrt{\log(k/\alpha)}\leq \bm{x}_{\ell}^T\overline{\bm{v}}\leq 2s\sqrt{\log(k/\alpha)}\right\}\\
    M^1_{s,\bm{v},\alpha} = \frac{1}{\vert \mathcal{A}_{s,\overline{\bm{v}},\alpha}\vert}\sum_{(\bm{x}_{\ell},z_{\ell})\in \mathcal{A}_{s,\overline{\bm{v}},\alpha}}(z_{\ell}-\bm{v}^T\bm{x}_{\ell})\\
    M^2_{s,\bm{v},\alpha} = \frac{1}{\vert \mathcal{A}_{s,\overline{\bm{v}},\alpha}\vert}\sum_{(\bm{x}_{\ell},z_{\ell})\in \mathcal{A}_{s,\overline{\bm{v}},\alpha}}(z_{\ell}-\bm{v}^T\bm{x}_{\ell})_+^2.
\end{gather*}
In words, $\mathcal{A}_{s,\overline{\bm{v}},\alpha}$ is the set of samples satisfying the event $A_{s,\overline{\bm{v}},\alpha}$, and $M^1_{s,\bm{v},\alpha}$ and $M^2_{s,\bm{v},\alpha}$ are the empirical estimates for $\mathbb{E}[z-\bm{x}^T\bm{v}\vert \mathcal{A}_{s,\overline{\bm{v}},\alpha}]$ and $\mathbb{E}[(z-\bm{x}^T\bm{v})_+^2\vert \mathcal{A}_{s,\overline{\bm{v}},\alpha}]$, respectively.

\begin{algorithm}[hbt!]
\caption{Sample-Efficient Regression with Self-Selection Bias}\label{alg:alg_one}
\LinesNumbered
\KwIn{Subspace $V$ such that $\max_{i\in [k]}\|P_{V}(\w_i)-\w_i\|\leq c_3\gamma/2t\sqrt{\log(k/\delta)}$}
\KwOut{Vectors $\widetilde{\w}_1,\ldots,\widetilde{\w}_k$ such that $\max_{i\in [k]} \|\w_i-\widetilde{\w}_i\|\leq \varepsilon$ with probability $1-\delta$}
Construct a $\left(\frac{c_3\gamma}{2t\sqrt{\log(k/\delta)}}\right)$-net $\mathcal{H}$ of $B_V(\Delta,B)$;

Draw $m_{\mathsf{moments}}$ (see \Cref{thm:total_samples}) i.i.d. samples $(\bm{x}_{\ell},z_{\ell})\in \mathbb{R}^n\times \mathbb{R}$ from \Cref{eq:lin_reg_ss};

Initialize $\mathcal{C}\gets \emptyset$;

\For{each $\bm{v}\in \mathcal{H}$}{Compute $M^1_{t,\bm{v},\delta}$, $M^1_{4t,\bm{v},\delta},M^2_{t,\bm{v},\delta}$ via samples;

If $\max\left\{\left\vert M^1_{t,\bm{v},\delta}\right\vert,\left\vert M^1_{4t,\bm{v},\delta}\right\vert\right\}\leq 5\gamma/8$, add $\bm{v}$ to $\mathcal{C}$.
\label{line:deletion_mean}
}

\While{\label{line:deletion_loop}$\mathcal{C}\neq \emptyset$}{Select
\begin{equation*}
    \bm{v}^*=\arg\min_{\bm{v}\in \mathcal{C}} M^2_{t,\bm{v},\delta}.
\end{equation*}

Let $\mathcal{S}=\{\bm{u}\in \mathcal{C}: \|\bm{v}^*-\bm{u}\|\leq 2\varepsilon\}$.

\For{$\bm{w}\in \mathcal{C}\setminus \mathcal{S}$}{
\If{\label{line:proj_loop}there exists $\bm{v}'\in \mathcal{S}$ such that $\|P_{\overline{\bm{w}}}(\bm{v}')-\bm{w}\|\leq 2\gamma/t\sqrt{\log(k/\delta)}$}{Delete $\bm{w}$ from $\mathcal{C}$.}
}

Delete $\mathcal{S}$ from $\mathcal{C}$.
\label{line:delete_proj}
}

Return all selected vectors $\bm{v}$.

\end{algorithm}

We first prove a number of important statements about the guarantees of \Cref{alg:alg_one} supposing the empirical moments are sufficiently accurate. Our first result leverages the structural results of the previous section to argue about the set of vectors that kept after \Cref{line:deletion_mean}.
\begin{lem}
\label{lem:assignment}
    Suppose that for all $\bm{v}\in \mathcal{H}$ it holds that
    \begin{gather}
    \label{eq:m1_close}
        \vert M^1_{t,\bm{v},\delta}-\mathbb{E}[z-\bm{x}^T\bm{v}\vert \mathcal{A}_{t,\overline{\bm{v}},\delta}]\vert\leq \gamma/8\\
        \label{eq:m1_close_4}
        \vert M^1_{4t,\bm{v},\delta}-\mathbb{E}[z-\bm{x}^T\bm{v}\vert \mathcal{A}_{4t,\overline{\bm{v}},\delta}]\vert\leq \gamma/8.
    \end{gather}
    Then at the beginning of \Cref{line:deletion_loop}, for all $\bm{v}\in \mathcal{C}$, at least one of the following two cases holds:
    \begin{enumerate}
    \item There exists a unique $j\in [k]$ such that
    \begin{equation*}
        \|\bm{v}-\w_j\|\leq \varepsilon.
    \end{equation*}
    \item There exists a unique $j\in [k]$ such that
    \begin{equation*}
        \|P_{\overline{\bm{v}}}(\w_j)-\bm{v}\|\leq \frac{\gamma}{t\sqrt{\log(k/\delta)}}.
    \end{equation*}
    \end{enumerate}

    Moreover, for every $j\in [k]$, every vector $\bm{v}\in \mathcal{H}$ such that $\|\bm{v}-\w_j\|\leq \frac{c_3\gamma}{t\sqrt{\log(k/\delta)}}$ is in $\mathcal{C}$ at the beginning of \Cref{line:deletion_loop}, and such a vector exists.
\end{lem}
\begin{proof}
    Suppose that $\bm{v}\in \mathcal{C}$ at the beginning of \Cref{line:deletion_loop} is such that the first item is false, so that $\bm{v}$ is $\varepsilon$-far from all $\w_1,\ldots,\w_k$ (note that by the triangle inequality, if $\varepsilon< \Delta^2/2B$, there is at most a single $\w_j$ that is $\varepsilon$-close by \Cref{lem:sep}). Since it was added to $\mathcal{C}$ at \Cref{line:deletion_mean}, \Cref{eq:m1_close} and \Cref{eq:m1_close_4} imply that 
    \begin{equation*}
        \max\left\{\left\vert \mathbb{E}\left[z-\bm{v}^T\bm{x}\bigg\vert A_{t,\overline{\bm{v}},\delta}\right]\right\vert,\left\vert \mathbb{E}\left[z-\bm{v}^T\bm{x}\bigg\vert A_{4t,\overline{\bm{v}},\delta}\right]\right\vert\right\}\leq 3\gamma/4.
    \end{equation*}
    Hence, the first case of \Cref{thm:separation} fails, and so (the proof of) \Cref{thm:separation} implies the existence of a unique $j\in [k]$ such that
    \begin{equation*}
        \|P_{\overline{\bm{v}}}(\w_j)-\bm{v}\|\leq \frac{\gamma}{t\sqrt{\log(k/\delta)}},
    \end{equation*}
    as claimed.

    The last claim holds just like in \Cref{thm:separation} again as a consequence of \Cref{cor:conditional_moments} with \Cref{eq:m1_close} and \Cref{eq:m1_close_4}, noting that $\mathcal{H}$ was taken sufficiently fine and our subspace $V$ is assumed to be sufficiently close to $\mathsf{span}(\w_1,\ldots,\w_k)$ to ensure that there exists $\bm{v}\in \mathcal{H}$ satisfying this condition by the triangle inequality.
\end{proof}

We write $\mathsf{assign}(\bm{v})\in [k]$ to denote the unique index $j\in [k]$ corresponding to the true regressor $\w_j$ satisfying either case in \Cref{lem:assignment} for $\bm{v}$. \Cref{lem:assignment} shows that this is well-defined for any $\bm{v}\in \mathcal{C}$. We next prove a number of essential invariants of \Cref{alg:alg_one} throughout the iterations of the while-loop in \Cref{line:deletion_loop}.

\begin{lem}
\label{lem:while_loop}
    Under the conditions of \Cref{lem:assignment}, if in addition it holds that 
    \begin{equation}
        \label{eq:m2_close}
        \vert M^2_{t,\bm{v},\delta}-\mathbb{E}[(z-\bm{x}^T\bm{v})_+^2\vert \mathcal{A}_{t,\overline{\bm{v}},\delta}]\vert\leq \gamma/8,
    \end{equation}
    the following holds in every iteration of the while-loop in \Cref{line:deletion_loop}:
    \begin{enumerate}
        \item There exists an index $j^*\in [k]$ such that $\|\w_{j^*}-\bm{v}^*\|\leq \varepsilon$ where $\bm{v}^*$ is the selected vector.

        \item The set $\mathcal{S}$ contains a vector $\bm{u}^*\in \mathcal{C}$ such that $\|\bm{u}^*-\w_{j^*}\|\leq \frac{c_3\gamma}{t\sqrt{\log(k/\delta)}}$.

        \item All $\bm{u}\in \mathcal{C}$ such that $\mathsf{assign}(\bm{u})=j^*$ get deleted from $\mathcal{C}$.

        \item For each index $\ell\in [k]$ that has not been selected so far, every $\bm{u}\in \mathcal{C}$ satisfying $\|\w_{\ell}-\bm{u}\|\leq \varepsilon$ does not get deleted in this iteration.
    \end{enumerate}
\end{lem}
\begin{proof}
    We prove this result inductively. Suppose these conditions have held up to the start of the current iteration of the while-loop in \Cref{line:deletion_loop}. For (1), suppose for a contradiction that $\bm{v}^*$ satisfies $\min_{j\in [k]} \|\bm{v}^*-\w_j\|\geq \varepsilon$ and let $j^*=\mathsf{assign}(\bm{v}^*)$ by \Cref{lem:assignment}. By (3), $j^*$ cannot have been selected in the while-loop yet since all such vectors get deleted during that iteration. There must also exist $\bm{u}\in \mathcal{C}$ such that $\|\bm{u}-\w_{j^*}\|\leq \frac{c_3\gamma}{t\sqrt{\log(k/\delta)}}$ since no close vector to $\w_{j^*}$ could have been deleted by \Cref{lem:assignment} and (4). But since the first case of \Cref{thm:separation} does not hold as a consequence of \Cref{lem:assignment}, \Cref{eq:smaller_sm} and \Cref{eq:m2_close} imply that $M^2_{t,\bm{u},\delta} \leq M^2_{t,\bm{v}^*,\delta}-\gamma/4$, which contradicts the minimality of $\bm{v}^*$. This proves (1).

    For (2), since $j^*$ was not selected so far by (3), there must exist $\bm{u}^*\in \mathcal{C}$ such that $\|\bm{u}^*-\w_{j^*}\|\leq \frac{c_3\gamma}{t\sqrt{\log(k/\delta)}}$ since such a vector exists in $\mathcal{C}$ by \Cref{lem:assignment} and could not have been deleted yet by (4). By the triangle inequality and the fact (1) has been established inductively, it holds that $\|\bm{v}^*-\bm{u}^*\|\leq 2\varepsilon$, and hence $\bm{u}^*\in \mathcal{S}$.

    For (3), we now observe that if $\mathsf{assign}(\bm{w})=j^*$ but $\bm{w}\not\in \mathcal{S}$, only the second case of \Cref{lem:assignment} must hold. But then the second case of \Cref{thm:separation} (namely \Cref{eq:close_proj_v}) implies that for the preceding very close $\bm{u}^*\in \mathcal{S}$ to $\w_{j^*}$,
    \begin{equation*}
        \|\bm{w}-P_{\overline{\bm{w}}}(\bm{u}^*)\|\leq \frac{2\gamma}{t\sqrt{\log(k/\delta)}},
    \end{equation*}
    and hence $\bm{w}$ gets deleted from $\mathcal{C}$ in \Cref{line:proj_loop}.

    Finally, we show (4). Suppose that $i\in [k]$ has not been selected so far, and suppose that $\bm{w}\in \mathcal{C}$ is such that $\|\bm{w}-\w_i\|\leq \varepsilon$. Then for \emph{any} $\bm{u}\in \mathcal{S}$, we also have $\|\w_{j^*}-\bm{u}\|\leq 3\varepsilon$ by the triangle inequality. We then compute:
    \begin{align*}
        \|P_{\overline{\bm{w}}}(\bm{u})-\bm{w}\|&=\|(P_{\overline{\bm{w}}}(\bm{u})-P_{\overline{\bm{w}_i}}(\bm{u}))+(P_{\overline{\bm{w}_i}}(\bm{u})-P_{\overline{\bm{w}_i}}(\w_{j^*}))\\
        &+(P_{\overline{\bm{w}_i}}(\w_{j^*})-\w_i)+(\w_i-\bm{w})\|\\
        &\geq \|P_{\overline{\bm{w}_i}}(\w_{j^*})-\w_i\|-\|P_{\overline{\bm{w}}}(\bm{u})-P_{\overline{\bm{w}_i}}(\bm{u})\|\\
        &-\|P_{\overline{\bm{w}_i}}(\bm{u})-P_{\overline{\bm{w}_i}}(\w_{j^*})\|-\|\w_i-\bm{w}\|\\
        &\geq \|P_{\overline{\bm{w}_i}}(\w_{j^*})-\w_i\|-\frac{2}{\Delta}\|\w-\w_i\|\|\bm{u}\|-\|\bm{u}-\w_{j^*}\|-\|\w_i-\w\|\\
        &\geq \|P_{\overline{\bm{w}_i}}(\w_{j^*})-\w_i\|-4B\varepsilon/\Delta-4\varepsilon.
    \end{align*}
    Here, we use \Cref{lem:close_proj} with \Cref{lem:contractions} and the fact that $\w$ and $\w_i$ have norm at least $\Delta$. However, note that 
    \begin{equation*}
        \|P_{\overline{\bm{w}_i}}(\w_{j^*})-\w_i\|^2\geq \Delta^4/B^2.
    \end{equation*}
    by the proof of \Cref{lem:sep}. Thus, so long as $\varepsilon\leq \Delta^3/16B^2$, the above is lower bounded by $\Delta^2/2B\gg \gamma$, and hence $\w$ is not deleted by any $\bm{u}\in \mathcal{S}$ in \Cref{line:delete_proj}. This proves (4) and thus the proof of the invariants.
\end{proof}

With \Cref{lem:while_loop}, we can finally prove the correctness of \Cref{alg:alg_one} conditional on all conditional moments being computed to sufficient accuracy.
\begin{thm}
\label{thm:correctness}
    Suppose that for all $\bm{v}\in \mathcal{H}$ it holds that
    \begin{gather}
        \vert M^1_{t,\bm{v},\delta}-\mathbb{E}[z-\bm{x}^T\bm{v}\vert \mathcal{A}_{t,\overline{\bm{v}},\delta}]\vert\leq \gamma/8\\
        \vert M^1_{4t,\bm{v},\delta}-\mathbb{E}[z-\bm{x}^T\bm{v}\vert \mathcal{A}_{4t,\overline{\bm{v}},\delta}]\vert\leq \gamma/8\\
        \vert M^2_{t,\bm{v},\delta}-\mathbb{E}[(z-\bm{x}^T\bm{v})_+^2\vert \mathcal{A}_{s,\overline{\bm{v}},\delta}]\vert\leq \gamma/8.
    \end{gather}
    Then \Cref{alg:alg_one} returns vectors $\widetilde{\w_1},\ldots,\widetilde{\w_k}$ such that
    \begin{equation*}
        \max_{i\in [k]}\|\widetilde{\w_i}-\w_i\|\leq \varepsilon.
    \end{equation*}
\end{thm}
\begin{proof}
    Under the conditions of the theorem, \Cref{lem:while_loop} implies that at each iteration of \Cref{line:deletion_loop}, we select a vector $\bm{v}^*$ that is $\varepsilon$-close to some regressor $\w_j$ for a new index $j$, and $\mathcal{C}$ always contains $\varepsilon$-close vectors to any $\w_i$ such that $i$ has not been selected yet. Because \Cref{lem:while_loop} also shows we never select the same assignment in $\mathsf{assign}$ twice since all such vectors get deleted in the iteration when the index is first selected, this implies we select exactly one $\varepsilon$-close vector for each $\w_j$, and no other vectors, as claimed. Thus \Cref{alg:alg_one} returns $\widetilde{\w}_1,\ldots,\widetilde{\w}_k$ with the stated guarantee.
\end{proof}

To complete the proof of correctness, we now simply must show that given sufficient samples, all moment statistics are computed to the desired accuracy. Using our previous results, it will hold that every conditional statistic we compute will have suitably bounded subexponential norm. Therefore, all these conditional moments will be computable to high accuracy using a polynomial in $k$ number of samples by the general Bernstein inequality. On the other hand, we will have many samples satisfying \emph{every} $A_{t,\bm{v},\delta}$ \emph{simultaneously} using just a polynomial number of samples by \Cref{cor:enough_thresh} since each of these events have at least inverse polynomial probability under the Gaussian measure by \Cref{prop:close_maximal}. Putting these facts together, we can obtain sufficiently accurate additive approximations to every statistic we need using only polynomial in $k$ samples. It is crucial at this point that exploit the subexponential concentration of all statistics we require to efficiently do a union bound over the exponential (in $k$) size net. We now state and prove this sample complexity bound:

\begin{thm}
\label{thm:total_samples}
    Let $\lambda\in (0,1)$. Define the following parameters for $C>0$ sufficiently large:
    \begin{gather*}
        K:=B^6\log(k/\delta)/\Delta^4\\
        m_{\mathsf{est}}=\frac{CK^2(k\log(\frac{CB^3\sqrt{\log(k/\delta)}}{\Delta^2\gamma})+\log(2/\lambda))}{\gamma^2}\\
        m_{\mathsf{moments}}= C\left(\frac{k}{\delta}\right)^{\Theta(B^4/\Delta^4)}\cdot m_{\mathsf{est}}.
    \end{gather*}
    
    Then given at least $m_{\mathsf{moments}}$ samples from the model in \Cref{eq:lin_reg_ss}, it holds with probability at least $1-\lambda$ that for all $\bm{v}\in \mathcal{H}$,
    \begin{gather*}
        \vert M^1_{t,\bm{v},\delta}-\mathbb{E}[z-\bm{x}^T\bm{v}\vert \mathcal{A}_{t,\overline{\bm{v}},\delta}]\vert\leq \gamma/8\\
        \vert M^1_{4t,\bm{v},\delta}-\mathbb{E}[z-\bm{x}^T\bm{v}\vert \mathcal{A}_{4t,\overline{\bm{v}},\delta}]\vert\leq \gamma/8\\
        \vert M^2_{t,\bm{v},\delta}-\mathbb{E}[(z-\bm{x}^T\bm{v})_+^2\vert \mathcal{A}_{s,\overline{\bm{v}},\delta}]\vert\leq \gamma/8.
    \end{gather*}
\end{thm}

\begin{proof}
    First, note that by standard bounds on the size of nets (see e.g. Lemma 5.13 of \citet*{vanhandel}), there is a constant $C>0$ such that 
    \begin{equation*}
        \vert \mathcal{H}\vert\leq \left(\frac{CB^3\sqrt{\log(k/\delta)}}{\Delta^2\gamma}\right)^k.
    \end{equation*}

    For any $\bm{v}$ such that $\|\bm{v}\|\leq B$ and any $\delta>0$, \Cref{lem:cond_sg} and  \Cref{lem:sg_to_se} imply that for any $t=\Theta(B^2/\Delta^2)$, the random variables
    \begin{gather*}
        Y\triangleq (\max_{j\in [k]} \bm{x}^T\w_j+\eta_j)-\bm{v}^T\bm{x}\\
        Y_+^2 \triangleq \left((\max_{j\in [k]} \bm{x}^T\w_j+\eta_j)-\bm{v}^T\bm{x}\right)_+^2
    \end{gather*}
    are both $K:=O(B^6\log(k/\delta)/\Delta^4)$-subexponential on the event $A_{t,\overline{\bm{v}},\delta}$. Therefore, by the General Bernstein inequality of \Cref{thm:bernstein}, to obtain a $\gamma/8$ additive approximation to each  $\mathbb{E}[Y\vert A_{t,\overline{\bm{v}},\delta}],\mathbb{E}[Y\vert A_{4t,\overline{\bm{v}},\delta}],$ and $\mathbb{E}[Y_+^2\vert A_{t,\overline{\bm{v}},\delta}]$ over all $\bm{v}\in \mathcal{H}$ simultaneously with probability $1-\lambda/2$, it suffices to compute the empirical averages for each with at least
    \begin{equation*}
        m_{\mathsf{est}} = CK^2\log(\vert \mathcal{H}\vert/\lambda)/\gamma^2=\frac{CK^2(k\log(\frac{CB^3\sqrt{\log(k/\delta)}}{\Delta^2\gamma})+\log(2/\lambda))}{\gamma^2}
    \end{equation*}
    samples and taking a union bound.

    In order to have this many samples that satisfy each event $A_{t,\bm{\overline{v}},\delta}$, recall from \Cref{prop:close_maximal} that $\Pr(A_{t,\bm{\overline{v}},\delta})\geq \frac{1}{k^{\Theta(B^4/\Delta^4)}}$. \Cref{cor:enough_thresh} then implies that so long as we are given $m'$ samples where $m'$ is at least
    \begin{equation*}
        m_{\mathsf{moments}}= C\left(\frac{k}{\delta}\right)^{\Theta(B^4/\Delta^4)}\cdot m_{\mathsf{est}}
    \end{equation*}
    samples $(\bm{x}_{\ell},z_{\ell})_{\ell=1}^{m'}$from the model \Cref{eq:lin_reg_ss}, then we will have at least $m_{\mathsf{est}}$ samples that satisfy each event $A_{t,\bm{\overline{v}},\delta}$ for all $\bm{v}\in \mathcal{H}$. This completes the proof.
\end{proof}

Putting together the sample complexity of \Cref{thm:subspace} (with $\varepsilon$ set to $c_3\gamma/2t\sqrt{\log(k/\delta)}$) to obtain a suitable $k$-dimensional subspace $V$ and the sample complexity of estimating moments in \Cref{thm:total_samples} for \Cref{alg:alg_one} gives an overall sample complexity bound to correctly output $\varepsilon$-approximations to $\w_1,\ldots,\w_k$ with probability at least $1-\lambda$ by \Cref{thm:correctness}. The formal statement is as given below as an immediate consequence of \Cref{thm:subspace},  \Cref{thm:correctness}, and \Cref{thm:total_samples}:

\begin{thm}[\Cref{thm:main_intro}, formal]    \label{thm:final_statement}
Under \Cref{assumption:uncovered} and \Cref{assumption:bounded}, there exists constants $C,c,c_1,c_2,c_3>0$ such that given any $\varepsilon>0$ such that $\varepsilon<c\Delta^4/B^4$ and failure probability $\lambda\in (0,1)$, the following holds. Define the following parameters:

\begin{gather*}
    t = CB^2/\Delta^2,\\
    \gamma = \frac{c_1\varepsilon^4}{B^2\log(B/c_1\varepsilon)},\\
    \delta = \frac{c_2\gamma^2}{B^4t^4\log^2(Btk/c_2\gamma)},\\
    \varepsilon' = c_3\gamma/2t\sqrt{\log(k/\delta)},\\
    m_{\mathsf{cov}}=\frac{CB^{34}k^{\Theta(B^4/\Delta^4)}\log^4(knB\log(k)/\varepsilon')\cdot (n+\sqrt{\log(2/\lambda)})}{\log^2(k)\Delta^{24}\varepsilon'^4},\\
    K=B^6\log(k/\delta)/\Delta^4,\\
m_{\mathsf{est}}=\frac{CK^2(k\log(\frac{CB^3\sqrt{\log(k/\delta)}}{\Delta^2\gamma})+\log(2/\lambda))}{\gamma^2},\\
    m_{\mathsf{moments}}= C\left(\frac{k}{\delta}\right)^{\Theta(B^4/\Delta^4)}\cdot m_{\mathsf{est}}.
\end{gather*}

Then given at least $m_{\mathsf{cov}}+m_{\mathsf{moments}}$ samples from the linear regression with self-selection bias model \Cref{eq:lin_reg_ss}, it holds with probability at least $1-\lambda$ that using $m_{\mathsf{moments}}$ samples for \Cref{alg:alg_one} with the low-dimensional subspace $V$ obtained via \Cref{thm:subspace} using $m_{\mathsf{cov}}$ samples returns $\widetilde{\w}_1,\ldots,\widetilde{\w}_k\in \mathbb{R}^n$ such that
\begin{equation*}
    \max_{i\in [k]}\|\widetilde{\w}_i-\w_i\|\leq \varepsilon.
\end{equation*}

In particular, for any fixed $B,\Delta>0$, the sample complexity is $\tilde{O}(n)\cdot \mathsf{poly}(k,1/\varepsilon,\log(1/\lambda))$, while the time complexity is $\mathsf{poly}(n,k,1/\varepsilon,\log(1/\lambda))+\left(\frac{\log(k)}{\varepsilon}\right)^{O(k)}$. If $k=O(1)$ is also a fixed constant, then the time complexity is $\mathsf{poly}(n,1/\varepsilon,\log(1/\lambda))$ while the sample complexity is $\tilde{O}(n)\cdot \mathsf{poly}(1/\varepsilon,\log(1/\lambda)).$
\end{thm}

\subsection{Tightness of Our Results}

We now comment on the tightness of the sample complexity bounds we obtain. First, we explain why under \Cref{assumption:uncovered} and \Cref{assumption:bounded} with fixed $B,\Delta>0$, the sample complexity dependence on $k$ of $k^{\Theta(B^4/\Delta^4)}$ is tight. We then further comment on the necessity of some form of subgaussianity assumption as in \Cref{assumption:bounded} for obtaining sample complexity with $\mathsf{poly}(k)$ dependence.
\begin{rmk}[Tightness on $k$ Dependence]
\label{rmk:k_tight}
    \normalfont We now discuss the tightness of our results in the dependence on $k$. For the special case of $k=2$, \citet{DBLP:conf/stoc/CherapanamjeriD23} provide a fully polynomial time algorithm in $n,1/\varepsilon,B,1/\Delta$ for isotropic Gaussian noise, whereas our sample complexity bounds pay exponentially in $B,1/\Delta$. However, this dependence is actually necessary for $k\gg2$ under \Cref{assumption:uncovered} and \Cref{assumption:bounded}. The case that $k=2$ is somewhat special in that so long as the regressors are not colinear (as implied by \Cref{assumption:uncovered}), the probability that either regressor is the maximum is at least $1/4$ if there is no noise, with no dependence on $B,\Delta$. This is simply because the linear form for one of the regressors will be positive while the other will be negative on the intersection of the event that two specific independent Gaussians attain positive and negative sign, respectively. Thus, in the $k=2$ case, either regressor will attain the maximum with constant probability without any dependence on $B,\Delta$ so long as $\Delta>0$.

    When $k\gg 2$, though, a simple example shows that one necessarily requires $k^{\Omega(B^4/\Delta^4)}$ samples to even observe a given regressor under \Cref{assumption:uncovered} and \Cref{assumption:bounded}. To see this, consider $\w_1=(B/2+2\Delta^2/B)\bm{e}_1$ while $\w_j = (B/2)\bm{e}_1+(B/2)\bm{e}_j$ for $j=2,\ldots,k$ and where there is no noise. One can easily verify that this collection of vectors satisfies \Cref{assumption:uncovered} and \Cref{assumption:bounded}. Observe that if $\bm{x}\sim \mathcal{N}(0,I)$, $(B/2)\max_{j>2} x_j$ will exceed $\Omega(B\sqrt{\log(k)})$ with all but $\exp(-k^{\Omega(1)})\ll k^{-B^4/\Delta^4}$ probability for fixed $B,\Delta>0$. On this event, the only way $\w_1$ is the maximizer is if $\Delta^2 x_1/B\gg B\sqrt{\log(k)}$, which will occur with probability at most $1/k^{\Omega(B^4/\Delta^4)}$ by \Cref{lem:gaussian_exact}. Thus, the overall probability $\w_1$ is the maximizer is at most $1/k^{\Omega(B^4/\Delta^4)}$, and so it takes at least $k^{\Omega(B^4/\Delta^4)}$ samples to even \emph{observe} a sample that was generated by $\w_1$ with constant probability. Because $B/\Delta\geq 1$, this implies that the sample complexity of \Cref{thm:final_statement} is essentially tight in the dependence of $k$, which is also of the form $k^{O(B^4/\Delta^4)}$ after collecting terms. Formally, there exists examples of \Cref{eq:lin_reg_ss} with $k-1$ and $k$ regressors whose observations can be coupled so that the first $k^{\Omega(B^4/\Delta^4)}$ observations are identical. Thus, one cannot identify the identity of all regressors, or even how many there are, with fewer samples. Note that our algorithm only requires an upper bound on $k$ to succeed, not the precise value.
    \end{rmk}

\begin{rmk}[Necessity of Subgaussianity]
\label{rmk:sg_tight}
    \normalfont Relatedly, some version of subgaussianity $\bm{\eta}$ appears necessary for similar reasons. If the maximum of the noise components is $\omega(\sqrt{\log(k)})$ with high probability (as can hold with weaker tail assumptions, like subexponentiality), then less than a $1/k^{\omega(1)}$ fraction of the samples will be determined by the actual regression values $\bm{x}^T\w_i$ since their maximum is at most $\Theta(B\sqrt{\log(k)})$ with high probability over $\bm{x}\sim \mathcal{N}(0,I)$. In this case, \Cref{eq:lin_reg_ss} becomes a small perturbation of a pure noise model, and only a subpolynomial (in $k$) fraction of the data will be informative of the underlying regressors.
    \end{rmk}

\section{Conclusion}
Our work yields a novel and substantially more general algorithm for solving \Cref{eq:lin_reg_ss}, which captures multiple existing models of linear regression with self-selection. This more refined analysis in terms of low-degree conditional moments leads to essentially optimal sample dependence on $k$, the number of latent linear models, under \Cref{assumption:uncovered} and \Cref{assumption:bounded}. Moreover, the simplicity of our approach provides improved runtime dependencies than prior work, despite being agnostic to the precise noise model. It would be very interesting to develop more efficient and direct approaches to extracting approximate regressors from the low-dimensional subspace without the need to brute force over sufficiently fine nets using, for instance, methods from continuous optimization or an appropriate filtering procedure. It would also be very interesting to go beyond the linear observation setting to simultaneously learn both the underlying, low-dimensional covariance structure as well as more general functions of the self-selection process. Our techniques rely on the ability to localize the observations to regions that suitably `explain' the observations (in a somewhat subtle sense) in a sample-efficient manner; finding analogues of this procedure for more general classes of functions is an intriguing direction for future research.
\newpage
\bibliography{bibliography}
\end{document}